%% file: main.tex
\documentclass[11pt,twoside]{article}

\usepackage{amsmath,amsthm,amsfonts,amssymb,mathtools,mathdots,scalerel,mathrsfs,stmaryrd,cmll}
\usepackage[utf8]{inputenc}
\usepackage[T1]{fontenc}
\usepackage{mlmodern, tikz-cd, quiver}
\usepackage[many]{tcolorbox}

\usepackage[shortlabels]{enumitem}
\usepackage[hidelinks]{hyperref}
\usepackage{tikz, fancyhdr, xparse, xcolor, tocloft}
\usepackage[british]{babel}
\usepackage[a4paper, top=4.5cm, bottom=4.5cm, left=4cm, right=4cm, asymmetric]{geometry}

\usepackage{crossreftools}
\usepackage[textwidth=3cm, textsize=small, colorinlistoftodos]{todonotes}

\usepackage{cctitle, titlesec}
\usepackage{etoolbox}

\usetikzlibrary{nfold}

\makeatletter
\patchcmd{\ttlh@hang}{\parindent\z@}{\parindent\z@\leavevmode}{}{}
\patchcmd{\ttlh@hang}{\noindent}{}{}{}
\makeatother

\input{macros}

\input{style}

\newcommand\runtitle{gray products of diagrammatic $(\infty, n)$-categories}
\newcommand\runauthor{chanavat}

\title{Gray products of diagrammatic $(\infty, n)$-categories}

\author{Cl\'emence Chanavat}

\institution{Tallinn University of Technology}

\begin{document}

\maketitle
\begin{center}
	\begin{minipage}[t]{.95\textwidth}
		\small\textsc{Abstract.}
			For each \( n \in \mathbb{N} \cup \set{\infty} \), diagrammatic sets admit a model structure whose fibrant objects are the diagrammatic \( (\infty, n) \)\nbd categories.
			They also support a notion of Gray product given by the Day convolution of a monoidal structure on their base category.     
			The goal of this article is to show that the model structures are monoidal with respect to the Gray product. 
			On the way to the result, we also prove that the Gray product of any cell and an equivalence is again an equivalence.
			Finally, we show that tensoring on the left or the right with the walking equivalence is a functorial cylinder for the model structures, and that the functor sending a diagrammatic set to its opposite is a Quillen self-equivalence.
	\end{minipage}
	
	\vspace{20pt}

	\begin{minipage}[t]{0.95\textwidth}
		\setcounter{tocdepth}{2}
		\tableofcontents
	\end{minipage}
\end{center}

\makeaftertitle

\input{intro}
\input{background}

\input{diagram}

\input{context}
\input{model}

\input{monoidal}

\bibliographystyle{alpha}
\small\bibliography{main.bib}

\end{document}

%% file: macros.tex
\newcommand\eqdef{\coloneqq}
\newcommand\nbd{\nobreakdash-\hspace{0pt}}
\newcommand\idd[1]{\mathrm{id}_{#1}}
\newcommand\bigid[1]{\mathrm{Id}_{#1}}

\newcommand\after{\circ}

\newcommand\incl{\hookrightarrow}
\newcommand\surj{\twoheadrightarrow}
\newcommand\restr[2]{{#1}{\raisebox{0pt}{$|_{#2}$}}}
\newcommand\powerset[1]{\mathscr{P}{#1}}

\newcommand\set[1]{\left\{ {#1} \right\}}
\newcommand\order[2]{#2^{(#1)}}

\newcommand\opp[1]{#1^\mathrm{op}}

\newcommand\cat[1]{\mathbf{#1}}

\newcommand\fun[1]{\mathsf{#1}}
\DeclareMathOperator*{\colim}{colim}

\DeclareMathOperator*{\Ob}{Ob}

\newcommand\rdcpx{\cat{RDCpx}}
\newcommand{\atom}{{\scalebox{1.3}{\( \odot \)}}}

\newcommand{\Set}{\cat{Set}}
\newcommand{\dgmSet}{\atom\Set}
\newcommand{\pt}{\mathbf{1}}
\newcommand{\init}{\varnothing}
\DeclareMathOperator{\Arr}{Arr}

\newcommand\clset[1]{\mathrm{cl}\set{#1}}
\newcommand\gr[2]{#2_{#1}}
\newcommand\maxel[1]{\mathscr{M}\!\mathit{ax}\,#1}
\newcommand\bd[2]{\partial_{#1}^{#2}}

\newcommand\faces[2]{\Delta_{#1}^{#2}}
\newcommand\inter[1]{\mathrm{int}\,#1}

\newcommand\cp[1]{\,{\scriptstyle\#}_{#1}\,}
\newcommand\cpsub[1]{\triangleright_{#1}}
\newcommand\subcp[1]{\prescript{}{#1\!}{\triangleleft}\;}
\newcommand\gencp[1]{\,{\scriptstyle\widehat{\#}}_{#1}\,}

\newcommand\subs[3]{#1[#2/#3]}
\newcommand\compos[1]{\langle#1\rangle}

\newcommand\celto{\Rightarrow}
\newcommand\relcelto[1]{\Rightarrow_{\kern-3pt #1}}

\newcommand\submol{\sqsubseteq}

\newcommand\dual[2]{\fun{D}_{#1}{#2}}
\newcommand\join{\,{\star}\,}
\newcommand\gray{\otimes}
\newcommand\sus[1]{\fun{S}{#1}}

\newcommand{\tp}{\otimes}

\DeclareMathOperator{\Pd}{Pd}
\DeclareMathOperator{\Rd}{Rd}
\DeclareMathOperator{\Eqv}{Eqv}
\DeclareMathOperator{\cell}{cell}
\DeclareMathOperator{\eqv}{eqv}
\DeclareMathOperator{\Dgn}{Dgn}
\DeclareMathOperator{\dgn}{dgn}
\DeclareMathOperator{\Nd}{Nd}
\DeclareMathOperator{\nd}{nd}

\newcommand\rev[1]{{#1}^\dag}

\newcommand{\un}{\varepsilon}
\newcommand{\lun}[1]{\lambda_{#1}}
\newcommand{\run}[1]{\rho_{#1}}
\newcommand{\hinv}[1]{\xi_{#1}}

\newcommand{\lcyl}[1]{\mathrm{L}_{#1}}
\newcommand{\rcyl}[1]{\mathrm{R}_{#1}}

\newcommand{\hcyl}[1]{\Xi_{#1}}

\newcommand\arr{\vec{I}}

\newcommand{\mapel}[1]{\iota_{#1}}
\newcommand{\imel}[2]{{#1}_{#2}}

\newcommand{\loc}[2]{#1[#2^{-1}]}
\newcommand{\preloc}[2]{#1\set{{#2}^{-1}}}
\newcommand{\selfloc}[1]{\tilde{#1}}
\newcommand{\locarr}{\tilde I}
\DeclareMathOperator{\Loc}{\fun{Loc}}
\newcommand{\revarr}{\tilde I}
\newcommand{\selflocm}[1]{\selfloc{#1}_{\m{}}}
\newcommand{\markarr}{\markmol{\arr}}

\renewcommand{\a}{\alpha}

\newcommand{\B}{\mathcal{B}}
\DeclareMathOperator{\Linv}{\mathcal{L}}
\DeclareMathOperator{\Rinv}{\mathcal{R}}
\DeclareMathOperator{\Pst}{Pst}

\renewcommand{\th}{\vartheta}

\newcommand{\F}{\fun{F}}
\newcommand{\G}{\fun{G}}

\newcommand{\m}[1]{{\mathsf{m}#1}}
\newcommand{\markmol}[1]{{#1}_{\m{}}}

\newcommand{\mdgmSet}{\atom\Set^{\m{}}}
\newcommand{\minmark}[1]{{#1}^{\flat}}
\newcommand{\maxmark}[1]{{#1}^{\sharp}}
\newcommand{\natmark}[1]{{#1}^{\natural}}

\newcommand{\minbdmap}[1]{\minmark{\partial}_{#1}}
\newcommand{\markbdmap}[1]{\partial^{\m{}}_{#1}}

\newcommand{\pp}[1]{\mathbin{\hat{#1}}}
\DeclareMathOperator{\an}{an}
\newcommand{\cof}[1]{{#1}\text{\nbd}\mathsf{cof}}
\newcommand{\pizero}{\pi_0}

\newcommand\cls[1]{\mathscr{#1}}
\newcommand{\C}{\cls C}

\newcommand{\Jn}[1]{J_{#1}}
\newcommand{\Jhorn}{J_{\mathsf{horn}}}
\newcommand{\Jcoind}{J_{\mathsf{coind}}}
\newcommand{\Jloc}{J_{\mathsf{loc}}}

%% file: style.tex
\newtheoremstyle{ittheorem}
  {\topsep}   % ABOVESPACE
  {\topsep}   % BELOWSPACE
  {\itshape}  % BODYFONT
  {0pt}       % INDENT (empty value is the same as 0pt)
  {\bfseries} % HEADFONT
  { ---}         % HEADPUNCT
  {5pt plus 1pt minus 1pt} % HEADSPACE
  {}          % CUSTOM-HEAD-SPEC

\newtheoremstyle{itdfn}
  {\topsep}
  {\topsep}
  {}
  {0pt}
  {\bfseries}
  {}
  {5pt plus 1pt minus 1pt}
  {\thmnumber{#2}{\thmnote{\normalfont\ \ %
{\sffamily(#3)}.}}}

\newtheoremstyle{itrmk}
  {0.5\topsep}
  {0.5\topsep}
  {\normalfont}
  {0pt}
  {\sffamily \itshape}
  { --- }
  {5pt plus 1pt minus 1pt}
  {}
  
\theoremstyle{ittheorem}
\newtheorem{thm}{Theorem}[section]
\newtheorem*{thm*}{Theorem}
\newtheorem{prop}[thm]{Proposition}
\newtheorem*{prop*}{Proposition}

\newtheorem{lem}[thm]{Lemma}

\theoremstyle{itdfn}
\newtheorem{dfn}[thm]{}
\theoremstyle{itrmk}
\newtheorem{rmk}[thm]{Remark}
\newtheorem{comm}[thm]{Comment}

\setlist{leftmargin=20pt,itemsep=0pt,parsep=0pt,topsep=1ex}

\titleformat{\section}
 {\Large\bfseries}{\thesection}{1em}{}

\titleformat{\subsection}
 {\normalsize\bfseries}{\thesubsection.}{1em}{}

%% file: intro.tex
\section*{Introduction}

In the study of higher structures, one is often led to consider a refined version of the cartesian product, known as the Gray tensor product and usually written \( - \gray - \). 
From its introduction in \cite{gray1974gray} for 2-categories, it has been successfully generalised to higher versions, notably to strict \( \omega \)\nbd categories \cite{AlAgl1993nerve, crans1995gray}, and several models of weak higher categories \cite{verity2008complicial,gagna2021gray,campion2023gray, abellan2023gray}. 
To informally understand what sorts of refinement the Gray product operates, consider the arrow category \( \arr \), with two objects and one non-trivial morphism between them. 
We depict side by side the cartesian product and the Gray product of the arrow category with itself:
\begin{center}
    \begin{tikzcd}[sep=small]
        & \bullet && \bullet && \bullet && \bullet \\
        {\arr \times \arr =} &&&& {\arr \gray \arr =} \\
        & \bullet && \bullet && \bullet && {\bullet.} \\
        && {} &&&& {}
        \arrow[from=1-2, to=1-4]
        \arrow[from=1-2, to=3-2]
        \arrow[dashed, from=1-2, to=3-4]
        \arrow["{{,}}", from=1-4, to=3-4]
        \arrow[from=1-6, to=1-8]
        \arrow[from=1-6, to=3-6]
        \arrow[from=1-8, to=3-8]
        \arrow[from=3-2, to=3-4]
        \arrow[shorten <=11pt, shorten >=6pt, Rightarrow, from=3-6, to=1-8]
        \arrow[from=3-6, to=3-8]
    \end{tikzcd}
\end{center}
The cartesian product \( \arr \times \arr \) is a strictly commutative square, while the Gray product \( \arr \gray \arr \) gives a (co)lax commutative square. 
The previous picture makes apparent another difference between the two products: given higher structures \( \C \) and \( \cls{D} \), if \( u \) is an \( n \)\nbd cell of \( \cls{C} \) and \( v \) is an \( m \)\nbd cell of \( \cls{D} \), then \( u \times v \) is a \( \max\set{n, m} \)\nbd cell of \( \C \times \cls{D} \), whereas \( u \gray v \) is an \( (n + m) \)\nbd cell of \( \C \gray \cls{D} \). 
Another point of divergence is that the Gray product is \emph{not} symmetric. 
However, passing to the opposite (that is, reversing the direction of all odd dimensional cells) distributes with the Gray product by reversing the order of the factors, in that \( \opp{(\C \gray \cls D)} \) is equivalent to \( \opp{\cls{D}} \gray \opp{\C} \).
As another example, the Gray product of the 2-globe with the arrow
\begin{center}
    \begin{tikzcd}[row sep=small]
        &&&& \bullet && \bullet && \bullet \\
        \bullet && \bullet & \gray && {=} \\
        &&&& \bullet && \bullet && {\bullet.}
        \arrow[""{name=0, anchor=center, inner sep=0}, curve={height=12pt}, from=1-7, to=1-9]
        \arrow[""{name=1, anchor=center, inner sep=0}, curve={height=-12pt}, from=1-7, to=1-9]
        \arrow[""{name=2, anchor=center, inner sep=0}, curve={height=18pt}, shorten <=13pt, shorten >=13pt, Rightarrow, from=1-7, to=3-9]
        \arrow[""{name=3, anchor=center, inner sep=0}, curve={height=-18pt}, shorten <=13pt, shorten >=13pt, Rightarrow, dashed, from=1-7, to=3-9]
        \arrow[""{name=4, anchor=center, inner sep=0}, curve={height=12pt}, from=2-1, to=2-3]
        \arrow[""{name=5, anchor=center, inner sep=0}, curve={height=-12pt}, from=2-1, to=2-3]
        \arrow[from=3-5, to=1-5]
        \arrow[from=3-7, to=1-7]
        \arrow[""{name=6, anchor=center, inner sep=0}, curve={height=12pt}, from=3-7, to=3-9]
        \arrow[""{name=7, anchor=center, inner sep=0}, curve={height=-12pt}, dashed, from=3-7, to=3-9]
        \arrow[from=3-9, to=1-9]
        \arrow[shorten <=3pt, shorten >=3pt, Rightarrow, from=0, to=1]
        \arrow[shorten <=6pt, shorten >=6pt, Rightarrow, scaling nfold=3, from=2, to=3]
        \arrow[shorten <=3pt, shorten >=3pt, Rightarrow, from=4, to=5]
        \arrow[shorten <=3pt, shorten >=3pt, Rightarrow, dashed, from=6, to=7]
    \end{tikzcd}
\end{center}
corresponds to a right cylinder: the Gray product is monoidal closed and the cells of its left and right internal-hom are higher lax and oplax transformations between functors of higher structures.

The higher structures that this article is concerned with are the \emph{diagrammatic sets}, of which we briefly recall the main features. A diagrammatic set is a presheaf over the shape category \( \atom \), whose objects are called \emph{atoms}. 
An atom of dimension \( m \in \mathbb{N} \) represents the shape of a higher-categorical \( m \)\nbd cell with the property that, for all \( k \le m \), its input and output \( k \)\nbd boundary are regular CW models of the closed topological \( k \)\nbd ball. 
Atoms, together with their maps, are described combinatorially by means of oriented graded posets, and this description can be extended to a subclass of diagrammatic set, known as the \emph{regular directed complexes}. 
Among those we find all the ``globular pastings'' of atoms, constructed by means of pushouts along the boundaries, which we call the \emph{molecules}. 
The \emph{round molecules} are a subclass of molecules that can appear in the boundary of an atom, hence are regular CW models of the closed topological ball. 
Then, if \( X \) is a diagrammatic set, a \emph{pasting diagram} in \( X \) is a morphism with codomain \( X \), and whose domain is a molecule.
A pasting diagram is a \emph{round diagram} if the molecule is round, and a cell if it is an atom. 
Pasting diagrams can be pasted together by the universal property of the pushout, and satisfy all the equations of strict \( \omega \)\nbd categories. Furthermore, each diagrammatic set comes equipped with an internal notion of \emph{equivalence} defined by coinduction. We say that an \( n \)\nbd dimensional round diagram is an equivalence if it is invertible up to \( (n + 1) \)\nbd dimensional equivalences witnessing the invertibility equations. 

On the one hand in \cite{chanavat2024homotopy, chanavat2024model}, the author and Hadzihasanovic explore diagrammatic sets as a model for weak higher categories. 
For each \( n \in \mathbb{N} \cup \set{\infty} \), they introduce model structures whose fibrant objects are exactly the diagrammatic \( (\infty, n) \)\nbd categories, that is, the diagrammatic sets with the property that each round diagram is equivalent to a cell, and all cells of dimension \( > n \) are equivalences.  
On the other hand, diagrammatic sets support a natural notion of Gray product, given by a Day convolution \cite{day1970closed} of the Gray product at the level of atoms \cite{hadzihasanovic2024combinatorics}.
The main goal of this paper is to show that this Gray product interacts well with the model structures: the nature of the interaction is made precise with the definition of \emph{monoidal model category}. 
Our main theorem, answering affirmatively \cite[Conjecture 6.4]{chanavat2024model}, is as follows.
\begin{thm*}
    Let \( n \in \mathbb{N} \cup \set{\infty} \). 
    The \( (\infty, n) \)\nbd model structure on diagrammatic sets is monoidal with respect to the Gray product.
\end{thm*}
\noindent In this particular context, it is safe to translate this by saying that for all diagrammatic sets \( X \), the functors \( X \gray - \) and \( - \gray X \) are both left Quillen. We derive this Theorem as a corollary of a similar result for the \emph{coinductive} \( (\infty, n) \)\nbd model structure on \emph{marked} diagrammatic sets, which are diagrammatic sets together with a subset of marked cells. 
The fibrant objects of this model structure are the diagrammatic \( (\infty, n) \)\nbd categories such that a cell is marked if and only if it is an equivalence.
Thus, the marking serves as a propositional truncation of the infinite tower of data needed to witness an equivalence, and makes, for instance, the computation of a pushout-product of a cofibration and a generating acyclic cofibration much more manageable.
Since forgetting the marking induces a Quillen equivalence with the \( (\infty, n) \)\nbd model structure, with an adjunction which is strict monoidal with respect to the Gray products on diagrammatic sets and marked diagrammatic sets, we recover at no cost the desired result.       

In order to make those computations, we still need to understand the interaction between Gray product and equivalences, which is the second central aspect of the article.
In particular, we prove the following theorem, which is an analogue of \cite[Proposition 4.9]{ara2020monoidal} where the authors show that the \emph{folk model structure on strict \( \omega \)\nbd categories} \cite{lafont2010folk} is monoidal with respect to the Gray product. 
\begin{thm*}
    Let \( X, Y \) be diagrammatic sets. 
    Then
    \begin{equation*}
        \Eqv X \gray \Rd Y \cup \Rd X \gray \Eqv Y \subseteq \Eqv (X \gray Y).
    \end{equation*}
\end{thm*}
\noindent Here, \( \Rd X \) denotes the collection of round diagrams of \( X \), and \( \Eqv X \) its collection of round diagrams which are also equivalences. 
This exhibits another difference between the Gray product and the cartesian product: \( u \gray v \) is an equivalence as soon as \( u \) or \( v \) is an equivalence. 
Unlike the case of strict \( \omega \)\nbd categories where this claim is proved by hand for the case \( \dim u = 1 \), then iterated for cells of arbitrary dimensions, we prove our result uniformly for all round diagrams by leveraging the machinery of \emph{context and natural equivalence} developed in \cite{chanavat2024equivalences}.  
  
\subsection*{Preservation of \texorpdfstring{\( \omega \)}{ω}-equivalences}

A morphism of diagrammatic sets is an \( \omega \)\nbd equivalence if 
\begin{itemize}
    \item it is essentially surjective on \( 0 \)\nbd cells,
    \item for all parallel pairs of round \( k \)\nbd diagrams in the domain, it is essentially surjective on \( (k + 1) \)\nbd round diagrams between their image in its codomain. 
\end{itemize}
This definition is the formal analogue for diagrammatic sets of the concept of the same name for strict \( \omega \)\nbd categories. In \cite{ara2020monoidal}, the authors left open whether tensoring two \( \omega \)\nbd equivalences is again an \( \omega \)\nbd equivalence. 
We do not have the tools to answer positively the analogue question for diagrammatic sets, yet we propose the following strategy. 
In \cite{chanavat2024model}, the author and Hadzihasanovic proved that the weak equivalences between fibrant objects in the \( (\infty, n) \)\nbd model structure coincide with the \( \omega \)\nbd equivalences.
Furthermore, as an immediate consequence of a fibrant replacement proposed by Hadzihasanovic during private discussions, but whose well-definiteness is still work-in-progress, the \( \omega \)\nbd equivalences would be exactly the weak equivalences of the \( (\infty, \infty) \)\nbd model structure. 
Since every diagrammatic set is cofibrant, Ken Brown's Lemma would answer affirmatively the open question.

\subsection*{The case of the join}

We conclude the discussion by mentioning the \emph{join}, which is another standard monoidal product for higher structures, usually written \( \join \). The category of atoms also supports this operation representably, although the monoidal unit, which is the empty regular directed complexes, is not an atom. This makes the extension to diagrammatic sets via Day convolution slightly more involved. As a consequence, the join is not biclosed but only \emph{locally} biclosed (see \cite[5.7]{ara2020joint} for the precise definitions)  and the notion of monoidal model category, usually for a \emph{biclosed} monoidal product, needs to be adapted, as in \cite{ara2020monoidal}. The initial goal of the author was to also include the proof that the \( (\infty, n) \)\nbd model structure is monoidal with respect to the join. The approach taken was to compare the join \( X \join Y \) to the Gray product of the suspension \( \sus{X} \gray \sus{Y} \), extending the one defined at the level of the pasting diagrams, constructed by Hadzihasanovic in \cite[7.4.14]{hadzihasanovic2024combinatorics}. However, in the process of writing the arguments, we realised that the intermediate steps were different enough from the case of the Gray product than simply interchanging the symbols \( \gray \) and \( \join \), yet too close to not look redundant and artificially doubling the amount of work necessary. For these reasons, we decided to not include the study of the join in this paper, and leave it for further dedicated investigations. To close this remark, we mention to the interested reader that the slick approaches taken in \cite{ara2020monoidal, loubaton2024inductive} to compare the join and the Gray product do not apply in our case: they involve maps of regular directed complexes which are unfortunately not representable for the sole reason of not being a cartesian fibration of their underlying posets. 

\subsection*{Structure of the article}

In Section \ref{sec:background}, we recall some notions about monoidal model categories. None of the results or definitions we present are new, nor given in their highest level of generality, only in a format that makes it convenient for us to use in the article. 

Section \ref{sec:diagram} starts by reviewing the generalised pasting at the \( k \)\nbd boundary and the pasting at a submolecule defined in \cite[Section 7.1]{hadzihasanovic2024combinatorics}. 
Indeed, even if the Gray product preserves the pushouts diagrams used to define molecules, the resulting pushout diagrams are not in general pastings, but instances of these generalised pastings. 
In particular, we provide a more explicit formula for the generalised pasting (Lemma \ref{lem:formula_generalised_pasting}), and prove that the \( m \)\nbd boundaries of a generalised pasting at the \( k \)\nbd boundary for \( k < m \) are again generalised pastings (Lemma \ref{lem:boundary_generalised_pasting_is_generalised_pasting}). 
Then, we introduce the Gray product of regular directed complexes and describe how it interacts with pastings and boundaries via the generalised pasting. 
After preliminary remarks on diagrammatic sets and their pasting diagrams, we define the functor \( \opp{(-)} \), which sends a diagrammatic set to its opposite, as well as the Day convolution of the Gray product, and show that the two distribute accordingly (Proposition \ref{prop:gray_opp_opp_gray}). 
Finally, we introduce some families of unitors, and the coinductive notion of equivalences, which we show to be preserved under passing to the opposite (Lemma \ref{lem:equivalence_in_opposite}). We conclude the section by extending the coinductive proof method to ease a later proof that collections of round diagrams are equivalences.

In Section \ref{sec:context}, we study how the Gray product interacts with context and equivalences. 
First, we recall the notion of contexts and their natural equivalences and show that the classes of left and right contexts preserve, in a precise sense, having a left and right inverse respectively (Lemma \ref{lem:closure_left_right_context}).
Then, we provide in Lemma \ref{lem:distribution_lower_gen_pasting_gray} a formula that makes explicit how the Gray product distributes over pasting at a submolecule, and synthesise this result using contexts (Lemma \ref{lem:gray_preserve_left_right_pasting_contexts}).  
Next, we prove Theorem \ref{thm:equivalences_in_gray}, claiming that tensoring an equivalence and a round diagram produces an equivalence, and conclude the section by showing that Gray products preserve certain classes of contexts (Lemma \ref{lem:pushout_product_preserves_A_equations_gray}).

In Section \ref{sec:model}, we recall the category of marked diagrammatic sets, its first properties, the extension of opposite and Gray product from the diagrammatic sets, as well as the localisation functor, which freely invert all the marked cells of a marked diagrammatic set. 
After giving the definition of \( (\infty, n) \)\nbd category and showing that it is stable under passing to the opposite (Lemma \ref{lem:opp_of_infty_cat_is_infty_cat}), we recall a pseudo-generating set of acyclic cofibrations for the coinductive \( (\infty, n) \)\nbd model structure on marked diagrammatic sets, and conclude with the definition of the \( (\infty, n) \)\nbd model structure on diagrammatic sets.

Finally, in Section \ref{sec:monoidal}, we show that the model structures on diagrammatic sets and marked diagrammatic sets are monoidal with respect to the Gray product.
We start by proving the preliminary Lemmas \ref{lem:opposite_pp_product} and \ref{lem:opp_Jhorn} which take advantage of the distributivity of the the functor \( \opp{(-)} \) with the Gray product to reduce the amount of work necessary.
Then, we move on to Proposition \ref{prop:pp_marked_horn_is_marked_horn}, which asserts that a certain class of acyclic cofibrations, the marked horns, are stable under pushout-product with a cofibration. 
Since the marked horns are closely related to contexts, we deduce this result as a direct consequence of the results of Section \ref{sec:context}.
Then, we prove that the coinductive \( (\infty, n) \)\nbd model structure is monoidal with respect to the Gray product in Theorem \ref{thm:marked_model_structure_monoidal}.
This amounts to a collection of checks, all of which are easy consequences of the previous results, and in particular of Theorem \ref{thm:equivalences_in_gray} on tensoring with equivalences. 
As a corollary, we deduce in Theorem \ref{thm:gray_monoidal_dgmset} that the \( (\infty, n) \)\nbd model structure on diagrammatic sets is also monoidal for the Gray product.
We conclude the article with two applications: we prove that tensoring on the left or on the right with the walking equivalence of shape the arrow \( \arr \) is a functorial cylinder for the \( (\infty, n) \)\nbd model structure (Proposition \ref{prop:reversible_cylinder_is_cylinder}), and from that, show in Theorem \ref{thm:op_quillen_auto_eq} that the functor sending a diagrammatic set to its opposite is a Quillen self-equivalence.

\subsection*{Acknowledgements}

Many thanks to Amar Hadzihasanovic for helpful conversations and detailed feedback. 

%% file: background.tex
\section{Monoidal model categories} \label{sec:background}

\subsection{Model categories}

\begin{dfn} [Left and right lifting classes]
    Let \( \C \) be a category and \( S \) be a class of morphisms in \( \C \). 
    We denote by \( l(S) \) and \( r(S) \) the classes of morphisms that have, respectively, the left and right lifting property with respect to all morphisms in \( S \). 
    If \( \C \) has a terminal object \( \pt \), we say that an object \( X \) of \( \C \) has the right lifting property against \( S \) if the unique morphism \( X \to \pt \) does.
\end{dfn}
\noindent We write \( \cof S \) for the set \( l(r(S)) \).

\begin{lem} [Retract Lemma]\label{lem:retract_lemma}
    Let \( f, p \) and \( i \) be morphisms in a category \( \C \) such that \( f \) factors as \( p \after i \) and has the left lifting property against \( p \).
    Then \( f \) is a retract of \( i \). 
\end{lem}
\begin{proof}
    See for instance \cite[Lemma 1.1.9]{hovey2007model}.
\end{proof}

\begin{dfn} [Relative \( J \)\nbd cell complex]
    Let \( J \) be a set of morphisms in a cocomplete category \( \C \).
    A \emph{relative \( J \)\nbd cell complex} is a morphism of \( \C \) which can be constructed as a transfinite composition of pushouts of morphisms of \( J \). 
\end{dfn}

\begin{rmk}
    Any relative \( J \)\nbd cell complex belongs to \( \cof J \).
\end{rmk}

\begin{prop} [Small object argument]
    Let \( \C \) be a locally presentable category, and \( J \) a set of morphisms of \( \C \).
    Then every morphism of \( \C \) factors as \( p \after i \) where:
    \begin{enumerate}
        \item \( i \) is a relative \( J \)\nbd cell complex, and
        \item \( p \) has the right lifting property against \( J \).
    \end{enumerate}
\end{prop}
\begin{proof}
    See for instance \cite[Theorem 2.1.14]{hovey2007model}, where the smallness hypothesis on the set \( J \) is satisfied since \( \C \) is locally presentable.
\end{proof}

\begin{rmk} \label{rmk:characterise_l_r_J}
    By the small object argument and the Retract Lemma, the morphisms of \( \cof J \) are exactly the retracts of relative \( J \)\nbd cell complexes.
\end{rmk}

\begin{dfn} [Cellular model]
    Let \( \C \) be a cocomplete category, and \( S \) be a class of morphisms of \( \C \).
    A \emph{cellular model} for \( S \) is a set \( J \) of morphisms such that \( S = \cof J \).
\end{dfn}

\noindent Until the end of the section, we let \( \C \) be a locally presentable model category. 

\begin{dfn} [Cylinder]
    Let \( X \) be an object of \( \C \). 
    A \emph{cylinder on \( X \)} is the data of an object \( \fun{I}X \) of \( \C \) and morphisms
    \begin{equation*}
        (\iota^-, \iota^+) \colon X \amalg X \to \fun{I}X,\quad \quad \sigma\colon \fun{I}X \to X 
    \end{equation*}
    such that \( (\iota^-, \iota^+) \) is a cofibration, \( \sigma \) is a weak equivalence, and \( \sigma (\iota^-, \iota^+) \) is a codiagonal.
\end{dfn}

\begin{dfn} [Functorial cylinder]
    A \emph{functorial cylinder on \( \C \)} is an endofunctor \( \fun{I} \) on \( \C \) together with natural transformations
    \begin{equation*}
        (\iota^-, \iota^+) \colon \bigid{\C} \amalg \bigid{\C} \to \fun{I},\quad \quad \sigma\colon \fun{I} \to \bigid{\C} 
    \end{equation*}
    which form a cylinder in each component.
\end{dfn}

\begin{dfn} [Left homotopy]
    Let \( f, g \colon X \to Y \) be two parallel morphisms in \( \C \).
    A \emph{left homotopy from \( f \) to \( g \)} is the data of a cylinder \( (\fun{I}X, \iota^\a, \sigma) \) on \( X \), together with a morphism \( \beta \colon \fun{I} X \to Y \) such that \( \beta \after \iota^- = f \) and \( \beta \after \iota^+ = g \).
    We say that \( f \) is \emph{left homotopic} to \( g \), and write \( f \approx g \), if there exists a left homotopy from \( f \) to \( g \).
    This defines a relation \( \approx \) on \( \hom_\C(X, Y) \). 
\end{dfn}

\begin{rmk} \label{rmk:left_homotopy_equivalence}
    By \cite[Theorem 7.4.5]{hirschhorn2003model}, if \( X \) is cofibrant, \( \approx \) is an equivalence relation.
\end{rmk}

\begin{dfn}
    Let \( X, Y \) be two object of \( \C \) with \( X \) cofibrant.
    We denote by \( \pizero(X, Y) \) the set of equivalence classes of morphisms from \( X \) to \( Y \) up to left homotopy.
\end{dfn}

\begin{lem} \label{lem:homotopic_implies_homotopic_any_cylinder}
    Let \( X, Y \) be objects of \( \C \) such that \( X \) is cofibrant and \( Y \) is fibrant, \( \fun{I} X \) be a cylinder on \( X \), and \( f, g \colon X \to Y \) be two morphisms in \( \C \) such that \( f \approx g \).
    Then there exists a left homotopy \( \beta \colon \fun{I} X \to Y \) from \( f \) to \( g \).
\end{lem}
\begin{proof}
    This is \cite[Proposition 7.4.10]{hirschhorn2003model}.
\end{proof}

\begin{prop} \label{prop:weak_equiv_wrt_cylinder}
    Let \( f \colon X \to Y \) be a morphism of cofibrant objects in \( \C \).
    Then \( f \) is a weak equivalence if and only if for every fibrant object \( W \) of \( \C \), the induced function of sets \( f^* \colon \pizero(Y, W) \to \pizero(X, W) \) is a bijection.
\end{prop}
\begin{proof}
    See \cite[Theorem 7.8.6]{hirschhorn2003model}.
\end{proof}

\begin{dfn} [Pseudo-generating set of acyclic cofibrations]
    We say that a set \( J \) of acyclic cofibrations in \( \C \) is \emph{pseudo-generating} if, for all morphisms \( f \) of \( \C \) with fibrant codomain, \( f \) is a fibration if and only if it has the right lifting property against \( J \). 
\end{dfn}

\begin{rmk} \label{rmk:other_def_pseudo_generating}
    Equivalently, a set \( J \) of acyclic cofibrations in \( \C \) is pseudo-generating if
    \begin{itemize}
        \item for all objects \( X \) of \( \C \), \( X \) is fibrant if and only if it has the right lifting property against \( J \), and
        \item for all morphisms \( f \) of \( \C \) between fibrant objects, \( f \) is a fibration if and only if it has the right lifting property against \( J \).
    \end{itemize}
\end{rmk}

\begin{lem} \label{lem:properties_of_pseudo_generating}
    Let \( J \) be a pseudo-generating set of acyclic cofibrations in \( \C \), and \( i \) be an acyclic cofibration with fibrant codomain.
    Then \( i \in \cof J \). 
\end{lem}
\begin{proof}
    By the small object argument for the set \( J \), the morphism \( i \) factors as \( p \after j \) where \( p \) has the right lifting property against \( J \) and \( j \) is a relative \( J \)\nbd cell complex.
    Since \( i \) has fibrant codomain, so does \( p \), which is therefore a fibration. 
    Thus, \( i \) has the left lifting property against \( p \), and by the Retract Lemma, \( i \) is a retract of \( j \), thus belongs to \( \cof J \).
\end{proof}

\begin{lem} \label{lem:left_adjoint_is_Quillen_enough_pseudo_generating}
    Let \( \cls D \) be a model category, \( \F \colon \C \to \cls D \) be a functor preserving pushouts and transfinite compositions, \( J \) be a pseudo-generating set of acyclic cofibrations in \( \C \) and suppose that \( \F \) sends cofibrations to cofibrations and morphisms of \( J \) to acyclic cofibrations.
    Then \( \F \) sends acyclic cofibrations to acyclic cofibrations.
\end{lem}
\begin{proof}
    Since \( \F \) preserves pushouts and transfinite compositions, it preserves relative \( J \)\nbd cell complexes, and all functors preserve retracts.
    Thus, by Remark \ref{rmk:characterise_l_r_J}, \( \F(\cof J) \) is a class of acyclic cofibrations in \( \cls D \). 
    Let \( i \colon X \to Y \) be an acyclic cofibration of \( \C \). 
    By the small object argument for the set \( J \) applied to the unique morphism \( Y \to \pt \), there exists a morphism \( j \colon Y \to \tilde{Y} \) in \( \cof J \) such that \( \tilde{Y} \to \pt \) has the right lifting property against \( J \), that is, such that \( \tilde Y \) is fibrant. 
    Then, \( j \after i \) is an acyclic cofibration with fibrant codomain, and by Lemma \ref{lem:properties_of_pseudo_generating}, it belongs to \( \cof J \), so \( \F(j \after i) \) is an acyclic cofibration.
    Since \( j \) is in \( \cof J \), so is \( \F j \), thus by two-out-of-three, the cofibration \( \F i \) is acyclic.
    This concludes the proof.  
\end{proof}

\begin{lem} \label{lem:ken_brown_plus}
    Let \( \fun{F} \colon \C \to \cls D \) be a left Quillen functor.
    Then 
    \begin{enumerate}
        \item \( \fun{F} \) sends weak equivalences between cofibrant objects to weak equivalences.
        \item If \( \fun{F} \) is furthermore a Quillen equivalence, then \( \fun{F} \) reflects weak equivalences between cofibrant objects. 
    \end{enumerate}
\end{lem}
\begin{proof}
    See \cite[Lemma 1.1.12, Corollary 1.3.16]{hovey2007model}.
\end{proof}

\subsection{Monoidal model structures}

\noindent From now on, we suppose that the locally presentable model category \( \C \) is endowed with a biclosed monoidal structure \( (\C, \tp, I) \). 

\begin{dfn} [Pushout-product]
    Let \( f \colon X \to Y \) and \( f' \colon X' \to Y' \) be two morphisms in \( \C \).
    The \emph{pushout-product of \( f \) and \( f' \)} is the morphism
    \begin{equation*}
        f \pp{\tp} f' \colon Y \tp X' \coprod_{X \tp X'} X \tp Y' \to Y \tp Y'
    \end{equation*}
    induced by universal property of the pushout.
    This assignment extends to a bifunctor
    \begin{equation*}
        - \pp{\tp} - \colon \Arr(\C) \times \Arr(\C) \to \Arr(\C), 
    \end{equation*}
    where \( \Arr(\C) \) denotes the category of arrows.
\end{dfn}

\begin{lem} \label{lem:pp_preserve_cof}
    Let \( J \) and \( J' \) be two sets of morphisms in \( \C \). 
    Then
    \begin{equation*}
        \cof J \pp{\tp} \cof{J'} \subseteq \cof{(J \pp{\tp} J')}. 
    \end{equation*}
\end{lem}
\begin{proof}
    See for instance \cite[Lemma B.0.10]{henry2020model}.
\end{proof}

\begin{dfn} [Monoidal model category]
    We say that \( \C \) is \emph{monoidal with respect to \( - \tp - \)} if
    \begin{enumerate}
        \item \( - \tp - \) satisfies the \emph{pushout-product axiom}, that is, for all pairs of cofibrations \( i \) and \( i' \), the pushout-product \( i \pp{\tp} i' \) is a cofibration, which is acyclic if \( i \) or \( i' \) is acyclic.
        \item for all objects \( X \) of \( \C \), the functors \( - \tp X \) and \( X \tp - \) send all cofibrant replacements of the monoidal unit to weak equivalences. 
    \end{enumerate} 
\end{dfn}

\begin{rmk}
    In the case where the monoidal unit is already cofibrant, the second axiom is superfluous.
\end{rmk}

\begin{lem} \label{lem:monoidal_pseudo_generating}
    Let \( J \) be a pseudo-generating set of acyclic cofibrations.
    Suppose that
    \begin{itemize}
        \item every object of \( \C \) is cofibrant, and
        \item for all cofibrations \( i, i' \) of \( \C \), the pushout-product \( i \pp\tp i' \) is a cofibration, which is acyclic if either \( i \) or \( i' \) is in \( J \).
    \end{itemize}
    Then \( \C \) is a monoidal model category with respect to \( - \tp - \).
\end{lem}
\begin{proof}
    Since every object of \( \C \) is cofibrant, so is the monoidal unit.
    For an object \( X \) of \( \C \), denote by \( i_X \colon \init \to X \) the unique morphism from the initial object, which is a cofibration by assumption.
    Then, we may identify the functors \( X \tp - \) and \( - \tp X \) with \( i_X \pp{\tp} - \) and \( - \pp{\tp} i_X \) respectively.
    By assumption and two-out-of-three, \( X \tp - \) and \( - \tp X \) satisfy the hypothesis of Lemma \ref{lem:left_adjoint_is_Quillen_enough_pseudo_generating}.
    By two-out-of-three and since acyclic cofibrations are stable under pushouts, we deduce that the pushout-product axiom is satisfied.
\end{proof}

%% file: diagram.tex
\section{Diagrams in a diagrammatic set} \label{sec:diagram}

\subsection{Complements on regular directed complexes}

We refer to the introduction of \cite{chanavat2024equivalences} for a brief summary of the notation relative to atoms, molecules, and regular directed complexes. 
All the details are in the monograph \cite{hadzihasanovic2024combinatorics}. 

\begin{dfn} [Atom inclusion]
    Let \( P \) be a regular directed complex, and \( x \in P \).
    The \emph{atom inclusion from \( x \)} is the unique cartesian inclusion of regular directed complexes \( \mapel{x} \colon \imel{P}{x} \incl P \) with image \( \clset{x} \).
\end{dfn}

\begin{dfn} [The arrow]
    We let \( \pt \) be the point. 
    The \emph{arrow} is the atom \( \pt \celto \pt \), with underlying graded poset \( I \eqdef \set{0^- < 1 > 0^+} \) and orientations \( 0^\a \in \faces{}{\a} 1 \).
\end{dfn}

\begin{dfn} [Merger of a round molecule]
    Let \( U \) be a round molecule. 
    The \emph{merger of \( U \)} is the atom \( \compos{U} \eqdef \bd{}{-} U \celto \bd{}{+} U \).
\end{dfn}

\begin{dfn} [Duals]
    For each subset \( J \subseteq \mathbb{N} \setminus \set{0} \), there exists a functor \( \dual{J}{} \) on regular directed complexes which reverses the orientations of faces of \( j \)\nbd dimensional elements for each \( j \in J \).
    If \( P \) is a regular directed complex, \( \dual{J}{P} \) is called the \emph{\( J \)\nbd dual of \( P \)}.
    For \( n \in \mathbb{N} \), we write \( \dual{n}{} \) for \( \dual{\set{n}}{} \).
    If \( J \) is the set of odd numbers, then \( \dual{J}{} \) is simply denoted \( \opp{(-)} \).
\end{dfn}

\begin{rmk}
    The functors \( \dual{J}{} \) are involutions, and respect the classes of atoms, round molecules, and molecules.
\end{rmk}

\begin{dfn} [Rewritable submolecule]
    The class \emph{submolecule inclusions} is the smallest subclass of inclusions of molecules such that
    \begin{enumerate}
        \item all isomorphisms are submolecule inclusions;
        \item for all molecules \( U, V \), and all \( k \geq 0 \), if \( U \cp{k} V \) is defined, then \(  U \incl U \cp{k} V \) and \( V \incl U \cp{k} V \) are submolecule inclusions;
        \item the composite of two submolecule inclusions is a submolecule inclusion.
    \end{enumerate} 
    If \( \iota \colon U \incl V \) is a submolecule inclusion, we write \( \iota \colon U \submol V \), or simply \( U \submol V \) if \( \iota \) is clear from the context, and we say that \( \iota \) is \emph{rewritable} if \( \dim U = \dim V \) and \( U \) is round.
\end{dfn}

\begin{dfn} [Generalised pasting]
    Let \( U, V \) be molecules, \( k \in \mathbb{N} \), and let
    \begin{center}
        \begin{tikzcd}
            {U \cap V} & V \\
            U & {U \cup V}
            \arrow[hook, from=1-1, to=1-2]
            \arrow[hook', from=1-1, to=2-1]
            \arrow[hook', from=1-2, to=2-2]
            \arrow[hook, from=2-1, to=2-2]
            \arrow["\lrcorner"{anchor=center, pos=0.125, rotate=180}, draw=none, from=2-2, to=1-1]
        \end{tikzcd}
    \end{center}
    be a pushout of inclusions in the category of oriented graded posets.
    We say that \( U \cup V \) is a \emph{generalised pasting at the \( k \)\nbd boundary}, and we write \( U \gencp{k} V \) for \( U \cup V \), if
    \begin{enumerate}
        \item \( U \cap V \submol \bd{k}{+} U \) and \( U \cap V \submol \bd{k}{-} V \),
        \item \( \bd{k}{\a} (U \cup V) \) is a molecule for all \( \a \) in \( \in \set{-, +} \),
        \item \( \bd{k}{-} U \submol \bd{k}{-} (U \cup V) \) and \( \bd{k}{+} V \submol \bd{k}{+} (U \cup V) \)
    \end{enumerate}  
\end{dfn}

\begin{rmk}
    By \cite[Lemma 7.1.4]{hadzihasanovic2024combinatorics}, a generalised pasting \( U \gencp{k} V \) of \( U \) and \( V \) is a molecule, and the evident inclusions \( U \incl U \gencp{k} V \) and \( V \incl U \gencp{k} V \) are submolecule inclusions.
\end{rmk}

\begin{dfn} [Pasting at a submolecule]
    Let \( U, V \) be molecules and \( k \in \mathbb{N} \). 
    Given a submolecule inclusion \( \iota \colon \bd{k}{+} U \submol \bd{k}{-} V \), we let \( U \cpsub{k, \iota} V \) be the pushout in the category of oriented graded posets of the span of inclusions
    \begin{center}
        \begin{tikzcd}
            {\bd{k}{+} U} & {\bd{k}{-} V} & {V} \\
            U && {U \cpsub{k, \iota} V.}
            \arrow["\iota", hook, from=1-1, to=1-2]
            \arrow[hook', from=1-1, to=2-1]
            \arrow[hook, from=1-2, to=1-3]
            \arrow[hook', from=1-3, to=2-3]
            \arrow[hook, from=2-1, to=2-3]
            \arrow["\lrcorner"{anchor=center, pos=0.125, rotate=180}, draw=none, from=2-3, to=1-2]
        \end{tikzcd}
    \end{center}
    If \( \iota \colon \bd{k}{-} U \submol \bd{k}{+} V \), we define dually \( V \subcp{k, \iota} U \).
    We call \( U \cpsub{k, \iota} V \), or \( V \subcp{k, \iota} U \), the \emph{pasting at a the submolecule \( \iota \)}, and we omit \( k \) when \( k = \dim V - 1 \), as well as \( \iota \) when it is clear from the context.
\end{dfn}

\begin{rmk}
    By \cite[Section 7.1]{hadzihasanovic2024combinatorics}, \( U \cpsub{\iota, k} V \) and \( V \subcp{k, \iota} U \)  are molecules and instances of generalised pasting. 
\end{rmk}

\begin{lem} \label{lem:formula_generalised_pasting}
    Let \( U, V \) be molecules, \( k \in \mathbb{N} \) and \( U \gencp{k} V \) be a generalised pasting.
    Then \( (\bd{k}{-} (U \gencp{k} V) \subcp{k} U) \subcp{k} V \) and \( U \cpsub{k}(V \cpsub{k} \bd{k}{+} (U \gencp{k} V)) \) are well-defined an uniquely isomorphic to \( U \gencp{k} V \).
\end{lem}
\begin{proof}
    This is a direct consequence of \cite[Lemma 7.1.4]{hadzihasanovic2024combinatorics}.
\end{proof}

\begin{lem} \label{lem:boundary_generalised_pasting_is_generalised_pasting}
    Let \( U, V \) be molecules, \( k, n \in \mathbb{N} \) be integers such that \( k < n \), \( \a \in \set{-, +} \) and \( U \gencp{k} V \) be a generalised pasting. 
    Then the pushout square
    \begin{center}
        \begin{tikzcd}[sep=small]
            {U \cap V} & {\bd{k}{-} V = \bd{k}{-} \bd{n}{\a} V} & {\bd{n}{\a} V} \\
            {\bd{k}{+} U = \bd{k}{+} \bd{n}{\a} U} \\
            { \bd{n}{\a} U} && {\bd{n}{\a} U \cup \bd{n}{\a} V}
            \arrow[hook, from=1-1, to=1-2]
            \arrow[hook', from=1-1, to=2-1]
            \arrow[hook, from=1-2, to=1-3]
            \arrow[hook', from=1-3, to=3-3]
            \arrow[hook', from=2-1, to=3-1]
            \arrow[hook, from=3-1, to=3-3]
            \arrow["\lrcorner"{anchor=center, pos=0.125, rotate=180}, draw=none, from=3-3, to=1-2]
        \end{tikzcd}
    \end{center}
    is a generalised pasting at the \( k \)\nbd boundary, and maps isomorphically onto \( \bd{n}{\a} (U \gencp{k} V) \).
\end{lem}
\begin{proof}
    By a straightforward variation of \cite[Lemma 3.1.15]{hadzihasanovic2024combinatorics}, \( \bd{n}{\a} U \cup \bd{n}{\a} V \) maps isomorphically to \( \bd{n}{\a} (U \gencp{k} V) \).
    From this, using globularity on \( k < n \) and that \( U \gencp{k} V \) is a generalised pasting, we deduce directly that \( \bd{n}{\a} U \cup \bd{n}{\a} V \) is a generalised pasting that the \( k \)\nbd boundary.
\end{proof}

\begin{rmk} 
    In other words, \( \bd{n}{\a} (U \gencp{k} V) = \bd{n}{\a} U \gencp{k} \bd{n}{\a} V \), generalising the usual formula for usual pasting. 
\end{rmk}

\begin{dfn} [Gray product]
    Let \( P, Q \) be two regular directed complexes.
    The \emph{Gray product of \( P \) and \( Q \)} is the oriented graded poset \( P \gray Q \) whose 
    \begin{itemize}
        \item underlying graded poset is \( P \times Q \), and
        \item orientation is specified, for all \( \a \in \set{-, +} \), by 
        \begin{equation*}
            \faces{}{\a} (x, y) = \faces{}{\a} x \times \set{y} + \set{x} \times \faces{}{(-)^{\dim x}\a} y
        \end{equation*}
    \end{itemize}
\end{dfn}

\begin{rmk} \label{rmk:gray_preserve_submol}
    By \cite[Corollary 7.2.17]{hadzihasanovic2024combinatorics} the Gray product of two regular directed complexes is a regular directed complexes. 
    Furthermore, by \cite[Proposition 1.21]{chanavat2024homotopy} the Gray product is monoidal on the category of regular directed complexes and cartesian maps, with monoidal unit the terminal regular directed complex \( \pt \).
    This monoidal structure restricts to the category of atoms.
\end{rmk}

\begin{prop} \label{prop:formula_gray_product}
    Let \( U, V, U', V' \) be molecules. 
    Then
    \begin{enumerate}
        \item for all \( j, k \in \mathbb{N} \),
        \begin{align*}
            \bd{n}{-} (U \gray V) &= \bd{n}{-} (\bd{j}{-} U \gray V) \gencp{n - 1} \bd{n}{-} (U \gray \bd{n - j - 1}{(-)^j} V), \\
            \bd{n}{+} (U \gray V) &= \bd{n}{+} (U \gray \bd{n - j - 1}{(-)^{j + 1}} V ) \gencp{n - 1} \bd{n}{+} (\bd{j}{+}U \gray V),
        \end{align*}
        \item for all \( \a \in \set{-, +} \) and \( n \in \mathbb{N} \),
        \begin{equation*}
            \bd{n}{\a} (U \gray V) = \bigcup_{k = 0}^n \bd{k}{\a} U \gray \bd{n - k}{(-)^k\a} V,  
        \end{equation*}
        \item if \( U \submol U' \) and \( V \submol V' \), then \( U \gray V \submol U' \gray V' \),
        \item if \( U \) is a generalised pasting \( W \gencp{k} W' \) at the \( k \)\nbd boundary, then 
        \begin{equation*}
            U \gray V = (W \gray V) \gencp{k + \dim V} (W' \gray V),
        \end{equation*}
        and if \( V \) is a generalised pasting \( W \gencp{k} W' \) at the \( k \)\nbd boundary, then
        \begin{equation*}
            U \gray V = 
            \begin{cases}
                (U \gray W) \gencp{k + \dim V} (U \gray W') & \text{if } \dim U \text{ is even,} \\
                (U \gray W') \gencp{k + \dim V} (U \gray W) & \text{if } \dim U \text{ is odd.}
            \end{cases}
        \end{equation*}
    \end{enumerate}
\end{prop}
\begin{proof}
    The first three statements are part, or a direct consequence, of \cite[Corollary 7.2.10, Proposition 7.2.16]{hadzihasanovic2024combinatorics}, and the last one is a slight generalisation thereof, which we prove as follows.
    Suppose that \( W \gencp{k} W' \) is the generalised pasting along the inclusions \( \iota \colon W \cap W' \incl W \) and \( \iota' \colon W \cap W' \incl W' \). By \cite[Lemma 7.2.8]{hadzihasanovic2024combinatorics}, the square
    \begin{center}
        \begin{tikzcd}
            {(W \cap W') \gray V} & {W' \gray V} \\
            {W \gray V} & {(W \gencp{k} W') \gray V}
            \arrow["{\iota' \gray V}", hook, from=1-1, to=1-2]
            \arrow["{\iota \gray V}"', hook', from=1-1, to=2-1]
            \arrow[hook', from=1-2, to=2-2]
            \arrow[hook, from=2-1, to=2-2]
            \arrow["\lrcorner"{anchor=center, pos=0.125, rotate=180}, draw=none, from=2-2, to=1-1]
        \end{tikzcd}
    \end{center}
    is a pushout square.
    The first two hypothesis of the generalised pasting are directly verified by the third point of the statement.
    For the last one, we have
    \begin{align*}
        \bd{k + \dim V}{-} (W \gray V) &= (\bd{k}{-} W \gray V) \gencp{k + \dim V - 1} (W \gray \bd{}{(-)^k} V) \\
        &\submol (\bd{k}{-} (W \gencp{k} W') \gray V) \gencp{k + \dim V - 1} ((W \gencp{k} W') \gray \bd{}{(-)^k} V) \\
        &= \bd{k + \dim V}{-} ((W \gencp{k} W') \gray V).
    \end{align*}
    The case \( \bd{k + \dim V}{+} (W' \gray V) \submol \bd{k + \dim V}{+} (W \gencp{k} W') \gray V \) uses a dual argument.
    This proves that \(  (W \gencp{k} W') \gray V = (W \gray W') \gencp{k + \dim V} (W' \gray V) \).
    The other part of the proof is similar, with extra care to flip a sign depending on the parity of \( \dim V \).
\end{proof}

\noindent We conclude this section by describing families of surjective maps that will model degeneracies and higher invertors in a diagrammatic set. 

\begin{dfn}[Partial cylinder]
    Given a graded poset \( P \) and a closed subset \( K \subseteq P \), the \emph{partial cylinder on \( P \) relative to \( K \)} is the graded poset \( I \times_K P \) obtained as the pushout
    \begin{center}
        \begin{tikzcd}
            {I \times K} & K \\
            {I \times P} & {I \times_K  P}
            \arrow[two heads, from=1-1, to=1-2]
            \arrow[hook', from=1-1, to=2-1]
            \arrow["{(-)}", hook', from=1-2, to=2-2]
            \arrow["q", two heads, from=2-1, to=2-2]
            \arrow["\lrcorner"{anchor=center, pos=0.125, rotate=180}, draw=none, from=2-2, to=1-1]
        \end{tikzcd}  
    \end{center}
    in the category of posets.
    This is equipped with a canonical projection map \( \tau_K \colon I \times_K P \surj P \).
\end{dfn}

\begin{dfn}[Partial Gray cylinder]
	Let \( U \) be a regular directed complex and \( K \subseteq U \) a closed subset.
	The \emph{partial Gray cylinder on \( U \) relative to \( K \)} is the oriented graded poset \( \arr \gray_K U \) whose
    \begin{itemize}
        \item underlying graded poset is \( I \times_K U \), and
        \item orientation is specified, for all \( \a \in \set{+, -} \), by
        \begin{align*}
            \faces{}{\a}(x) & \eqdef \set{(y) \mid y \in \faces{}{\a}x}, \\
            \faces{}{\a}(i, x) & \eqdef \begin{cases}
                \set{(0^\a, x)} + \set{(1, y) \mid y \in \faces{}{-\a}x \setminus K} &
                \text{if \( i = 1 \),} \\
                \set{(i, y) \mid y \in \faces{}{\a}x \setminus K} + 
                \set{(y) \mid y \in \faces{}{\a}x \cap K} &
                \text{otherwise}.
            \end{cases}
        \end{align*}
    \end{itemize}
\end{dfn}

\begin{dfn}[Inverted partial Gray cylinder]
	Let \( U \) be a molecule, \( n \eqdef \dim U \), and \( K \subseteq \bd{}{+}U \) a closed subset.
	The \emph{left-inverted partial Gray cylinder on \( U \) relative to \( K \)} is the oriented graded poset \( \lcyl{K} U \) whose
    \begin{itemize}
        \item underlying graded poset is \( I \times_K U \), and
        \item orientation is as in \( \arr \gray_K U \), except for all \( x \in \gr{n}{U} \) and \( \a \in \set{+, -} \)
    \begin{align*}
        \faces{}{-}(1, x) &\eqdef \set{(0^-, x), (0^+, x)} + \set{(1, y) \mid y \in \faces{}{+}x \setminus K}, \\
        \faces{}{+}(1, x) &\eqdef \set{(1, y) \mid y \in \faces{}{-}x}, \\
        \faces{}{\a}(0^+, x) &\eqdef \set{(0^+, y) \mid y \in \faces{}{-\a}x \setminus K} + 
            \set{(y) \mid y \in \faces{}{-\a}x \cap K}.
    \end{align*}
\end{itemize}
	Dually, if \( K \subseteq \bd{}{-}U \), the \emph{right-inverted partial Gray cylinder on \( U \) relative to \( K \)} is the oriented graded poset \( \rcyl{K}{U} \) whose
    \begin{itemize}
        \item underlying graded poset is \( I \times_K U \), and
        \item orientation is as in \( \arr \gray_K U \), except for all \( x \in \gr{n}{U} \) and \( \a \in \set{+, -} \)
        \begin{align*}
            \faces{}{-}(1, x) &\eqdef \set{(1, y) \mid y \in \faces{}{+}x}, \\
            \faces{}{+}(1, x) &\eqdef \set{(0^-, x), (0^+, x)} + \set{(1, y) \mid y \in \faces{}{-}x \setminus K}, \\
            \faces{}{\a}(0^-, x) &\eqdef \set{(0^-, y) \mid y \in \faces{}{-\a}x \setminus K} + 
                \set{(y) \mid y \in \faces{}{-\a}x \cap K}.
        \end{align*}
    \end{itemize}
\end{dfn}

\begin{rmk} \label{rmk:inverted_cylinder_well_def}
	By \cite[Lemma 1.20, Lemma 1.26]{chanavat2024equivalences}, partial Gray cylinders and inverted partial Gray cylinders respect the classes of molecules, round molecules and atoms.
	Moreover, for all molecules \( U \) and closed subset \( K \subseteq U \),
	\begin{itemize}
		\item \( \tau_K \colon \arr \gray_K U \to U \) is a cartesian map of molecules,
		\item if \( p \colon U \to V \) is a cartesian map of molecules with \( \dim V < \dim U \), then \( p \after \tau_K \colon \lcyl{K}{U} \to V \) and \( p \after \tau_K \colon \rcyl{K}{U} \to V \) are cartesian maps of molecules.
	\end{itemize}
\end{rmk}

\begin{dfn} [Higher invertor shapes]
    Let \( U \) be a round molecule.
    The family of \emph{higher invertor shapes on \( U \)} is the family of molecules \( \hcyl s U \) indexed by strings \( s \in \set{L, R}^* \), defined inductively on the length of \( s \) by
    \begin{align*}
        \hcyl{\langle\rangle} U & \eqdef U, \\
        \hcyl{Ls} U &\eqdef \lcyl{\bd{}{+} \hcyl{s} U} (\hcyl{s} U), \\
                \hcyl{Rs} U &\eqdef \rcyl{\bd{}{-} \hcyl{s} U} (\hcyl{s} U).
    \end{align*}
	These are equipped with cartesian maps \( \tau_s \colon \hcyl{s} U \to U \) of their underlying posets, with the property that for all cartesian maps of molecules \( p \colon U \to V \) such that \( \dim V < \dim U \), the composite \( p \after \tau_s \) is a cartesian map of molecules.
\end{dfn}

\subsection{Diagrammatic sets}

\noindent We recall that a diagrammatic set is a presheaf on the category \( \atom \), of atoms and cartesian maps.
Furthermore, by \cite[Lemma 2.5]{chanavat2024homotopy}, the category \( \rdcpx \), of regular directed complexes and cartesian maps, embeds fully and faithfully into the category of diagrammatic sets. 
This allows us to make no distinction between a regular directed complex and its associated diagrammatic set. 

\begin{dfn} [Diagram in a diagrammatic set]
    Let \( U \) be a regular directed complex and \( X \) a diagrammatic set.
    A \emph{diagram of shape \( U \) in \( X \)} is a morphism \( u \colon U \to X \).
    A diagram is called
    \begin{itemize}
        \item a \emph{pasting diagram} if \( U \) is a molecule,
        \item a \emph{round diagram} if \( U \) is a \emph{round} molecule, and
        \item a \emph{cell} if \( U \) is an atom.
    \end{itemize}
    We write \( \dim u \eqdef \dim U \).
    We write respectively \( \Pd X \), \( \Rd X \) and \( \cell X \) the sets of pasting diagrams, round diagrams, and cells in \( X \).
\end{dfn}

\begin{rmk}
    Any atom is a round molecule, and any round molecule is a molecule, thus \(  \cell X \subseteq \Rd X \subseteq \Pd X \).
\end{rmk}

\begin{rmk}
    By the Yoneda Lemma, we identify a cell \( u \colon U \to X \) with its corresponding element \( u \in X(U) \).
    Furthermore, since isomorphisms of molecules are unique when they exists, we may safely identify isomorphic pasting diagrams in the slice over \( X \).
\end{rmk}

\begin{dfn}
    Recall that an \( \omega \)\nbd graph in degree \( \geq k \) is a graded set \( G = \sum_{n \geq k} \gr{n}{G} \), together with boundary functions \( \bd{}{-}, \bd{}{+} \colon \gr{n + 1}{G} \to \gr{n}{G} \) for each \( n \geq k \), satisfying \( \bd{}{\a} \bd{}{-} = \bd{}{\a} \bd{}{+} \) for all \( \a \in \set{-, +} \).
    More generally, for \( n \geq k \) and \( \a \in \set{-, +} \), we define inductively \( \bd{k}{\a} \colon \gr{n}{G} \to G_k \) by \( \bd{k}{\a} = \idd{G_k} \) if \( n = k \), and \( \bd{k}{\a} = \bd{}{\a}\bd{k + 1}{\a} \) if \( n > k \). 
    Given \( n > k \), and \( a \in \gr{n}{G} \), we write \( a \colon a^- \celto a^+ \) to mean \( \bd{}{\a} a = a^\a \), for \( \a \in \set{-, +} \), and say that \( a \) is of \emph{type} \( a^- \celto a^+ \).
    We say that \( a, b \in \gr{n}{G} \) are \emph{parallel} if they have the same type.
    Given parallel \( a, b \in G_k \), 
    \begin{equation*}
        G(a, b) \eqdef \set{c \in \sum_{n > k} \gr{n}{G} \mid \bd{k}{-} c = a, \bd{k}{+} c = b}
    \end{equation*}
    admits a structure of \( \omega \)\nbd graph in degree \( \geq k + 1 \), with grading and boundaries induced by \( G \). 
\end{dfn}

\begin{dfn}[\( \omega \)\nbd graph of pasting diagrams]
    Let \( u \colon U \to X \) be a pasting diagram in a diagrammatic set \( X \).
    For \( n \geq 0 \) and \( \a \in \set{-, +} \), we write \( \bd{n}{\a} u \) for the pasting diagram \( \restr{u}{\bd{n}{\a}U} \colon \bd{n}{\a} U \to X \).
    We may omit the index \( n \) if \( n = \dim u - 1 \).

    The set \( \Pd X \) is graded by dimension; given a subset \( A \subseteq \Pd X \), we write \( \gr{n}{A} \eqdef \set{u \in A \mid \dim u = n} \).
    Then, \( \Pd X \) admits the structure of \( \omega \)\nbd graph in degree \( \geq 0 \) with the functions \( \bd{}{-}, \bd{}{+} \colon \gr{n + 1}{\Pd X} \to \gr{n}{\Pd X} \).
    Grading and boundaries are compatible with the inclusion \( \Rd X \subseteq \Pd X \), making \( \Rd X \) an \( \omega \)\nbd subgraph in degree \( \geq 0 \) of \( \Pd X \).  
\end{dfn}

\begin{dfn} [Subdiagram]
    Let \( u \colon U \to X \) be a pasting diagram in a diagrammatic set \( X \).
    A \emph{subdiagram of \( u \)} is a pair of a pasting diagram \( v \colon V \to X \) and a submolecule inclusion \( \iota \colon V \incl U \) such that \( u \after \iota = v \).
    A subdiagram is \emph{rewritable} if the submolecule inclusion \( \iota \) is rewritable.
    We write \( \iota \colon v \submol u \) for the data of a subdiagram of \( u \), or simply \( v \submol u \) if \( \iota \) is irrelevant or evident from the context. 
\end{dfn}

\begin{dfn} [Pasting of pasting diagrams]
    Let \( u \colon U \to X \) and \( v \colon V \to X \) be pasting diagrams, such that \( \bd{k}{+} u = \bd{k}{-} v \).
    We let \( u \cp{k} v \colon U \cp{k} V \to X \) be the unique pasting diagram determined by the universal property of the pushout \( U \cp{k} V \).
    More generally, for any generalised pasting at the \( k \)\nbd boundary \( U \gencp{k} V \), we write \( u \gencp{k} v \colon U \gencp{k} V \to X \) for the pasting diagram determined by the universal property of the pushout \( U \gencp{k} V \).
    If \( U \gencp{k} V \) is the pasting at a submolecule \( U \cpsub{k, \iota} V \), then \( \iota \colon \bd{k}{+} u \submol \bd{k}{-} v \) is a subdiagram, and we write \( u \cpsub{k, \iota} v \) for the pasting diagram \( u \gencp{k} v \).
    Dually, if \( U \gencp{k, \iota} V \) is the pasting at a submolecule \( U \subcp{k, \iota} V \), then \( \iota \colon \bd{k}{-} v \submol \bd{k}{+} u \) is a subdiagram, and we write \( u \subcp{k, \iota} v \) for the pasting diagram \( u \gencp{k} v \). 
\end{dfn}
\noindent We often omit the index \( k \) when it is equal to \( \min \set{\dim u, \dim v} - 1 \), and omit \( \iota \) when it is irrelevant or evident from the context.

\begin{dfn} [Opposite of a diagrammatic set]
    Precomposing any diagrammatic set \( X \) with the functor \( \opp{(-)} \colon \atom \to \atom \) induces the involution
    \begin{equation*}
        \opp{(-)} \colon \dgmSet \to \dgmSet.
    \end{equation*}
    The diagrammatic set \( \opp{X} \) is called the \emph{opposite} of \( X \).
    For a subset \( A \subseteq \Pd X \), we write \( \opp{A} \eqdef \set{\opp{u} \in \Pd \opp{X} \mid u \in A } \).
\end{dfn}

\noindent Recall that given a small monoidal category, we can extend its monoidal structure via Day convolution \cite{day1970closed} to its category of presheaves.
\begin{dfn} [Gray product]
   Let \( X, Y \) be two diagrammatic sets.
   The \emph{Gray product of \( X \) and \( Y \)} is the diagrammatic set \( X \gray Y \) given by the Day convolution of the Gray product of atoms.
   This extends to a biclosed monoidal structure with product
   \begin{equation*}
    - \gray - \colon \dgmSet \times \dgmSet \to \dgmSet,
   \end{equation*}
   and monoidal unit \( \pt \), the terminal diagrammatic set.
\end{dfn}

\begin{rmk}
    By \cite[Corollary 2.8]{chanavat2024homotopy}, the embedding of regular directed complexes into diagrammatic sets is strong monoidal with respect to the Gray products of regular directed complexes and diagrammatic sets. 
\end{rmk}

\begin{rmk}
    Let \( \iota \colon u \submol v \) be a subdiagram in \( X \) and \( \iota' \colon u' \submol v' \) be a subdiagram in \( X' \).
    Then by Proposition \ref{prop:formula_gray_product}, \( \iota \gray \iota' \colon u \gray u' \submol v \gray v' \) is a subdiagram in \( X \gray X' \).
\end{rmk}

\begin{rmk}
    Since \( \atom \) is an Eilenberg-Zilber category, by \cite[Lemma 3.3]{chanavat2024homotopy}, a non-degenerate cell in \( X \gray Y \) can always be chosen of the form \( u \gray v \), for \( u \) a non-degenerate cell of \( X \) and \( v \) a non-degenerate cell of \( Y \).    
    Furthermore, any morphism \( f \colon X \gray Y \to Z \) is uniquely determined by its values on cells of the form \( u \gray v \).
\end{rmk}

\begin{prop} \label{prop:gray_opp_opp_gray}
    Let \( X, Y \) be diagrammatic sets. 
    Then the morphism
    \begin{equation*}
        \begin{array}{ccc}
            \opp{(X \gray Y)} & \to     & \opp{Y} \gray \opp{X} \\
            \opp{(u \gray v)} & \mapsto & \opp{v} \gray \opp{u},
        \end{array}
    \end{equation*}
    is a natural isomorphism of diagrammatic sets.
\end{prop}
\begin{proof}
    The statement is true on atoms by \cite[Proposition 7.5.28]{hadzihasanovic2024combinatorics}, thus we conclude by universal property of Day convolution.
\end{proof}

\subsection{Degenerate diagrams and equivalences}

\begin{dfn} [Degenerate pasting diagram]
    Let \( u \colon U \to X \) be a diagram.
    We say that \( u \) is \emph{degenerate} if there exists a pair of a diagram \( v \colon V \to X \) and a surjective cartesian map of regular directed complexes \( p \colon U \to V \) such that \( v \after p = u \), and \( \dim v < \dim v \). 
    We let
    \begin{align*}
        \Dgn X &\eqdef \set{u \in \Pd X \mid u \text{ is degenerate}} & \dgn X \eqdef \Dgn X \cap \cell x,\\
        \Nd X &\eqdef \set{u \in \Pd X \mid u \text{ is not degenerate}} & \nd X \eqdef \Nd X \cap \cell x.
    \end{align*}
\end{dfn}

\begin{dfn} [Reverse of a degenerate diagram]
    Let \( u \colon U \to X \) be a degenerate diagram, equal to \( v \after p \) for some diagram \( v \colon V \to X \) and surjective cartesian map \( p \colon U \to V \) with \( n \eqdef \dim u > \dim v \).
    The \emph{reverse of \( u \)} is the diagram \( \rev{u} \eqdef v \after \dual{n}{p} \colon \dual{n}{U} \to X \).
\end{dfn}

\begin{dfn} [Unit]
    Let \( u \colon U \to X \) be a pasting diagram.
    The \emph{unit on \( u \)} is the degenerate pasting diagram \( \un u \colon u \celto u \) defined by \( u \after \tau_{\bd{}{} U} \colon \arr \gray_{\bd{}{}U} U \to X \).
\end{dfn}

\begin{dfn} [Left and right unitor]
    Let \( u \colon U \to X \) be a pasting diagram. 
    For \( \iota \colon v \submol \bd{}{-} u \) a rewritable subdiagram of shape \( V \), call \( K \eqdef \bd{}{} U \setminus \inter{\iota(V)} \).
    The \emph{left unitor of \( u \) at \( \iota \)} is the degenerate pasting diagram \( \lun{\iota} u \colon u \celto \un v \cpsub{\iota} u \) defined by \( u \after \tau_K \colon \arr \gray_K U \to X \).
    Dually, for \( \iota \colon v \submol \bd{}{+} u \) a rewritable subdiagram of shape \( V \), call \( K \eqdef \bd{}{} U \setminus \inter{\iota(V)} \).
    The \emph{right unitor of \( u \) at \( \iota \)} is the degenerate pasting diagram \( \run{\iota} u \colon u \subcp{\iota} \un v \celto u \) defined by \( u \after \tau_K \colon \arr \gray_K U \to X \).
\end{dfn} 
\noindent We simply write \( \lun{}u \) and \( \run{} u \) when \( \iota \) is an isomorphism.

\begin{rmk}
    If \( u \) is round, then by Remark \ref{rmk:inverted_cylinder_well_def}, so are \( \un u \), \( \lun{\iota} u \) and \( \run{\iota} u \).
\end{rmk}

\begin{dfn} [Equivalence]
    Let \( X \) be a diagrammatic set and \( u \colon a \celto b \) a round diagram in \( X \) with \( n \eqdef \dim u > 0 \).
    We say that \( u \) is an \emph{equivalence} if there exist round diagrams \( u^L, u^R \colon b \celto a \), together with two equivalences \( h \colon \un b \celto u^R \cp{} u \) and \( z \colon u \cp{} u^L \celto \un a \).
    The diagrams \( u^R \) and \( u^L \) are respectively called a \emph{right-inverse} and a \emph{left-inverse} of \( u \).  
    We let
    \begin{equation*}
        \Eqv X \eqdef \set{u \in \Rd X \mid u \text{ is an equivalence}},\quad\quad \eqv X \eqdef \Eqv X \cap \cell X.
    \end{equation*}
\end{dfn}

\begin{comm} \label{comm:coinductive_proof_method}
    This is a coinductive definition of equivalences, based on bi-invertibility.
    Let \( \Rinv \) be the operator on the power set \( \powerset{(\Rd X)} \) defined by sending \( A \subseteq \Rd X \) to the set of round diagrams \( u \colon a \celto b \) such that there exists \( u^R \colon b \celto a \) in \( \Rd X \) and \( h \colon \un b \celto u^R \cp{} u \) in \( A \).
    Dually, let \( \Linv \) be the operator sending \( A \subseteq \Rd X \) to the set of round diagrams \( u \colon a \celto b \) such that there exists \( u^L \colon b \celto a \) in \( \Rd X \) and \( z \colon u \cp{} u^L \celto \un a \) in \( A \).
    Then \( \Eqv X \) is the greatest fixed point of the operator \( \B \eqdef \Linv \cap \Rinv \).
    We get the corresponding coinductive proof method: if \( A \subseteq \Linv(A) \cap \Rinv(A) \), then \( A \subseteq \Eqv X \).
\end{comm}

\begin{dfn} [Equivalent round diagrams]
    Let \( u, v \) be a parallel pair of round diagrams in a diagrammatic set \( X \).
    We say that \emph{\( u \) is equivalent to \( v \)}, and write \( u \simeq v \), if there exists an equivalence \( e \colon u \celto v \) in \( X \).
\end{dfn}

\begin{thm} \label{thm:main_equivalences}
    Let \( X \) be a diagrammatic set.
    Then
    \begin{enumerate}
        \item all degenerate round diagrams are equivalences;
        \item for each \( n \geq 0 \), the relation \( \simeq \) is an equivalence relation on \( \gr{n}{\Rd X} \), and for \( n > 0 \) a congruence with respect to \( - \cpsub{n - 1} - \) and \( - \subcp{n - 1} - \),
        \item the set of equivalences is closed under \( \simeq \),
        \item the left and right inverse of an equivalence are both equivalences, and equivalent to each other.
    \end{enumerate}
    Moreover, any morphism of diagrammatic sets \( f \colon X \to Y \) sends equivalences to equivalences.
\end{thm}
\begin{proof}
    See \cite[Theorem 2.13, Proposition 2.17, Proposition 2.19, Proposition 2.31]{chanavat2024equivalences}.
\end{proof}

\begin{dfn} [Weak inverse]
    Let \( u \colon a \celto b \) be a round diagram in a diagrammatic set \( X \).
    A \emph{weak inverse of \( u \)} is any round diagram \( u^* \colon b \celto a \) such that \( \un b \simeq u^* \cp{} u \) and \( u \cp{} u^* \simeq \un a \). 
\end{dfn}

\begin{rmk}
    By Theorem \ref{thm:main_equivalences}, \( u \) has a weak inverse if and only if it is an equivalence, and in that case, a left or a right inverse of \( u \) is also a weak inverse.
\end{rmk}

\begin{lem} \label{lem:equivalence_in_opposite}
    Let \( X \) be a diagrammatic set. 
    Then \( \opp{(\Eqv X)} = \Eqv \opp{X} \).
\end{lem}
\begin{proof}
    We show that \( \opp{(\Eqv X)} \subseteq \B (\opp{(\Eqv X)}) \), by coinduction this will prove \( \opp{(\Eqv X)} \subseteq \Eqv \opp{X} \).
    Let \( \opp{u} \in \opp{(\Eqv X)} \), then \( u \) has a weak inverse \( v \). 
    Since \( \simeq \) is symmetric, regardless of the parity of \( \dim u \), we can always find equivalences \( h \) and \( z \) such that \( \opp{h} \) has type \( \un \bd{}{+} \opp{u} \celto \opp{v} \cp{} \opp{u} \) and \( \opp{z} \) has type \( \opp{u} \cp{} \opp{v} \celto \bd{}{-} \opp{u} \).
    Therefore, \( \opp{u} \in \B (\opp{(\Eqv X)}) \).
    Finally, applying the first part of the proof to \( \opp{X} \) yields \( \opp{(\Eqv \opp{X})} \subseteq \Eqv X \).
    Applying the functor \( \opp{(-)} \) to this inclusion shows that \( \Eqv \opp{X} \subseteq \opp{(\Eqv X)} \).
    This concludes the proof.
\end{proof}

\begin{lem} \label{lem:generalised_pasting_equivalences_is_equivalences}
    Let \( X \) be a diagrammatic sets, \( u, v \in \Eqv X \) be equivalences in \( X \) with \( k \eqdef \dim u = \dim v \), and \( u \gencp{k - 1} v \) be a generalised pasting.
    Then \( u \gencp{k - 1} v \) is an equivalence.
\end{lem}
\begin{proof}
    By Lemma \ref{lem:formula_generalised_pasting} and using a unitor,
    \begin{align*}
        u \gencp{k - 1} v &= u \cpsub{} (v \cpsub{k - 1} \bd{}{+} (u \gencp{k - 1} v)) \\
        &\simeq u \cpsub{} (v \cpsub{k - 1} \bd{}{+} (u \gencp{k - 1} v)) \cp{} \un (\bd{}{+} (u \gencp{k - 1} v)) \\
        &= u \cpsub{} (v \cpsub{} \un (\bd{}{+} (u \gencp{k - 1} v))).
    \end{align*}
    Then, \( u, v \) and \( \un (\bd{}{+} (u \gencp{k - 1} v)) \) are equivalences, hence by Theorem \ref{thm:main_equivalences} so are \( u \cpsub{} (v \cpsub{} \un (\bd{}{+} (u \gencp{k - 1} v))) \), and \( u \gencp{k - 1} v \).
\end{proof}

\begin{dfn} 
    Let \( X \) be a diagrammatic set. 
    Given \( A \subseteq \Rd X \), we let \( \Pst A \) be the closure of the set \( A \cup \Eqv X \) under the following clause: for all round diagrams \( u, v \in \Pst A \) of the same dimension, if a pasting \( u \cpsub{} v \), or \( v \subcp{} u \), is defined, then it belongs to \( \Pst A \). 
    Given a parallel pair \( u, v \), we write \( u \relcelto A v \) if there exists a round diagram \( h \colon u \celto v \) in \( \Pst A \).
\end{dfn}

\begin{rmk}
    A round diagram \( u \colon a \celto b \) is in \( \B(\Pst A) \) if and only if there exist round diagrams \( u^R, u^L \colon b \celto a \) such that \( \un b \relcelto A u^R \cp{} u \) and \( u \cp{} u^L \relcelto A \un a \).
\end{rmk}

\noindent We conclude this section on equivalences by extending the proof method of Comment \ref{comm:coinductive_proof_method} to the following: if \( A \subseteq \B(\Pst A) \), then \( A \subseteq \Eqv X \).

\begin{lem} \label{lem:celto_A_is_congruence_preorder}
    Let \( X \) be a diagrammatic set, \( A \subseteq \Rd X \), \( n \in \mathbb{N} \).
    Then for each \( n \geq 0 \), the relation \( \relcelto A \) 
    \begin{enumerate}
        \item contains the relation \( \simeq \),
        \item is a preorder relation on \( \gr{n}{\Rd X} \), and, for \( n > 0 \), is compatible with the operations \( - \cpsub{n - 1} - \) and \( - \subcp{n - 1} - \).
    \end{enumerate}
\end{lem}
\begin{proof}
    By construction, \( \Eqv X \subseteq \Pst(A) \) so that \( \relcelto A \) contains \( \simeq \).
    Then, the reflexivity of \( \simeq \) implies the one of \( \relcelto A \).
    For transitivity, if \( h \colon u \celto v \) and \( z \colon v \celto w \) are in \( \Pst(A) \), \( h \cp{} z \colon u \celto w \) is again in \( \Pst(A) \), so that \( u \relcelto{A} w \).
    Now suppose that \( u \relcelto A u' \) and \( v \relcelto A v' \) are round diagrams of the same dimension and that a pasting at a submolecule \( u \cpsub{} v \) is defined.
    By assumption, there are round diagrams \( h \colon u \celto u' \) and \( z \colon v \celto v' \) both in \( \Pst(A) \).
    Then, \( (\un (u \cpsub{} v) \subcp{} h) \subcp{} z \colon u \cpsub{} v \celto u' \cpsub{} v' \) is in \( \Pst(A) \). 
    This shows that \( u \cpsub{} v \relcelto A u' \cpsub{} v' \).
    The case where a pasting \( v \subcp{} u \) is defined is dual.
    This concludes the proof.
\end{proof}

\begin{prop} \label{prop:coinduction_on_closure_pst}
    Let \( X \) be a diagrammatic set, \( A \subseteq \Rd X \) be a subset of round diagrams such that \( A \subseteq \B(\Pst A) \).
    Then \( \Pst A \subseteq \Eqv X \).
\end{prop}
\begin{proof}
    We will prove, by structural induction, that for all \( w \in \Pst A \), we have \( w \in \B(\Pst A) \).
    By coinduction, this will prove \( \Pst A \subseteq \Eqv X \).
    The base case \( w \in A \cup \Eqv X \) holds by assumption on \( A \) and definition of \( \Eqv X \).
    We consider the case \( w = u \cpsub{\iota} v \colon a \celto b \), with \( u, v \in \B(\Pst A) \).
    Since \( u, v \in \Rinv(\Pst A) \), we have round diagrams \( v^R \) and \( u^R \) such that \( \un b \relcelto A v^R \cp{} v \) and \( \un\bd{}{+} u \relcelto{A} u^R \cp{} u \).
    Using the unitor \( \lun{\iota} v \), we also know that \( v^R \cp{} v \simeq v^R \cp{} ((\un \bd{}{+} u) \cpsub{} v) \).
    Thus Lemma \ref{lem:celto_A_is_congruence_preorder} shows that \( \un b \relcelto{A} (v^R \subcp{} u^R) \cp{} (u \cpsub{} v) \), that is, \( u \cpsub{\iota} v \in \Rinv(\Pst A) \).
    Next, since \( u, v \in \Linv(\Pst A) \), we have round diagrams \( v^L \) and \( u^L \) such that \( u \cp{} u^L \relcelto A \un \bd{}{-} u \) and \( v \cp{} v^L \relcelto A \un \bd{}{-} v \).
    Then,
    \begin{align*}
        (u \cpsub{} v) \cp{} (v^L \subcp{} u^L) 
            &=\hspace{6pt} u \cpsub{} ((v \cp{} v^L) \subcp{} u^L) \\
            &\relcelto A u \cpsub{} ((\un \bd{}{-} v) \subcp{} u^L) && \text{by Lemma \ref{lem:celto_A_is_congruence_preorder}} \\
            &\simeq\hspace{6pt} (\un a \subcp{} u) \cp{} ((\un \bd{}{-} v) \subcp{} u^L) && \text{by left unitor} \\
            &\simeq\hspace{6pt} (\un a \subcp{} u) \subcp{} u^L && \text{by right unitor} \\
            &=\hspace{6pt} \un a \subcp{} (u \cp{} u^L) \\
            &\relcelto A \un a \subcp{} \un \bd{}{-} u && \text{by Lemma \ref{lem:celto_A_is_congruence_preorder}} \\
            &\simeq\hspace{6pt} \un a && \text{by right unitor}.
    \end{align*} 
    By Lemma \ref{lem:celto_A_is_congruence_preorder}, \( u \cpsub{} v \in \Linv(\Pst A) \).
    Therefore \( u \cpsub{} v \in \B(\Pst A) \).
    The case where \( w = v \subcp{} u \) is dual, so this concludes the coinduction and the proof.
\end{proof}

%% file: context.tex
\section{Gray product of contexts} \label{sec:context}

\subsection{Contexts and natural equivalences}

\begin{dfn} [Context for pasting diagrams]
    Let \( X \) be a diagrammatic set.
    For \( k \) ranging over \( \mathbb N \) and \( v, w \) over parallel pairs in \( \gr{k}{\Pd X} \), the class of \emph{context on \( \Pd X(v, w) \)} is the inductive class of morphisms of \( \omega \)\nbd graphs in degree \( > k \) with domain \( \Pd X(v, w) \) generated by the following clauses.
    \begin{enumerate}
        \item (\textit{Identity}). The identity \( - \colon \Pd X(v, w) \to \Pd X(v, w) \) is a context on \( \Pd(v, w) \).
        \item (\textit{Left Pasting}). For all \( u \in \gr{k + 1}{\Rd X} \) and rewritable \( \iota \colon \bd{}{+} u \submol v \),
        \begin{equation*}
            u \cpsub{\iota} - \colon \Pd X(v, w) \to \Pd X(\subs{v}{\bd{}{-}u}{\iota(\bd{}{+}u)}, w)
        \end{equation*}
        is a context on \( \Pd X(v, w) \).
        \item (\textit{Right Pasting}). For all \( u \in \gr{k + 1}{\Rd X} \) and rewritable \( \iota \colon \bd{}{-} u \submol w \),
        \begin{equation*}
            - \subcp{\iota} u \colon \Pd X(v, w) \to \Pd X(v, \subs{w}{\bd{}{+}u}{\iota(\bd{}{-}u)})
        \end{equation*}
        is a context on \( \Pd X(v, w) \).
        \item (\textit{Composition}). If \( \F \colon \Pd X(v, w) \to \Pd X(v', w') \) and \( \G \) is a context on \( \Pd X(v', w') \), then the composite \( \G\F \) is a context on \( \Pd X(v, w) \).
        \item (\textit{Promotion}) If \( k > 0 \) and \( \F \) is a context on \( \Pd X(\bd{}{-}v, \bd{}{+} w) \), then
        \begin{equation*}
            \F_{v, w} \eqdef \restr{\F}{\Pd X(v, w)} \colon \Pd X(v, w) \to \Pd X(\F v, \F w)    
        \end{equation*}
        is a context on \( \Pd X(v, w) \).
    \end{enumerate}
    We let \( \dim \F \eqdef k + 1 \) be the \emph{dimension} of any context \( \F \) on \( \Pd X(v, w) \).
\end{dfn}

\begin{rmk}
    For every context \( \F \) on \( \Pd X(v, w) \) and diagram \( u \colon v \celto w \), there is an evident subdiagram \( u \submol \F u \).
\end{rmk}

\begin{dfn} [Context determined by a subdiagram]
    Let \( X \) be a diagrammatic set and \( \iota \colon u \submol v \) be a subdiagram with \( \dim u = \dim v > 0 \).
    The \emph{context determined by \( \iota \)} is unique context \( \F \colon \Pd X(\bd{}{-} u, \bd{}{+} u) \to \Pd X(\bd{}{-}v, \bd{}{+}v) \)
    such that the evident subdiagram \( u \submol \F u \) is equal to \( \iota \).
\end{dfn}

\begin{rmk}
    The context determined by a rewritable submolecule \( \iota \colon u \submol v \) is the context mapping any round diagram \( u' \colon \bd{}{-} u \celto \bd{}{+} u \) to the substitution \( \subs{v}{u'}{u} \).
\end{rmk}

\begin{rmk}
    If \( \F \) is a context on \( \Pd X(v, w) \), then for any diagram \( u \colon v \celto w \), the context determined by the evident subdiagram \( u \submol \F u \) is \( \F \) itself.
\end{rmk}

\begin{dfn} [Left and right context]
    The class of \emph{left contexts} (respectively \emph{right contexts}) is the inductive subclass of contexts generated without the clause left pasting (respectively right pasting).      
\end{dfn}

\begin{dfn} [Trim contexts]
    The class of \emph{trim contexts} is the inductive subclass of contexts generated without the clause promotion.
\end{dfn}

\begin{dfn} [\( A \)\nbd contexts]
    Let \( X \) be a diagrammatic set, and \( A \subseteq \Rd X \). 
    The class of \emph{\( A \)\nbd contexts} is the inductive subclass of contexts obtained by restricting the clauses left pasting and right pasting to \( u \in A \).
    If \( A = \Eqv X \), we speak of \emph{weakly invertible context}.
\end{dfn}

\begin{rmk}
    We may freely combine the terminology: for instance a weakly invertible left trim context is either the identity, or a composition of contexts of the form \( - \subcp{} u \), with \( u \in \Eqv X \). 
\end{rmk}

\begin{dfn} [Shape of a context]
    Let \( X \) be a diagrammatic set, \( v, w \) be parallel round diagrams of shape \( V \) and \( W \) respectively, and \( \F \) be a context on \( \Pd X(v, w) \).
    There exists a unique pair of a molecule \( U' \), together with an inclusion \( (V \celto W) \incl U' \) such that for all round diagrams \( u \colon v \celto w \) of shape \( U \), the pasting diagram \( \F u \) has a shape given by the substitution \( \subs{U'}{U}{(V \celto W)} \).
    Then, the molecule \( U' \) is called the \emph{shape of \( \F \)}.
\end{dfn}

\begin{dfn} [Round context]
    Let \( X \) be a diagrammatic set, \( v, w \) a parallel pair of round diagrams and \( \F \) a context on \( \Pd X(v, w) \).
    We say that \( \F \) is \emph{round} if its shape is round.
\end{dfn}

\begin{lem}\label{lem:decomposition_context}
    Let \( X \) be a diagrammatic set and \( \F \) be a context on \( \Pd(v, w) \) such that \( \dim \F > 1 \). 
    Then there exist 
    \begin{enumerate}
        \item a context \( \G \) on \( \Pd X(\bd{}{-}v, \bd{}{+}w) \), and
        \item a trim context \( \fun{T} \) on \( \Pd X(\G v, \G w) \)
    \end{enumerate}
    such that \( \F = \fun{T}\G_{v, w} \). Moreover, 
    \begin{enumerate}
        \item if \( \F \) is round, then \( \G \) and \( \fun{T} \) are round.
        \item if \( \F \) is left (respectively right), then \( \fun{T} \) and \( \fun{G} \) can be chosen left (respectively right).
        \item if \( \F \) is weakly invertible, then \( \fun{T} \) and \( \fun{G} \) can be chosen weakly invertible.
    \end{enumerate}
\end{lem}
\begin{proof}
    A straightforward variation of \cite[Lemma 3.3 and 3.13]{chanavat2024equivalences}.
\end{proof}

\begin{comm} \label{comm:extended_pasting}
    In light of Lemma \ref{lem:decomposition_context}, we can alternatively generate the class of contexts with the clauses identity, composition, and the extension of left pasting and right pasting to all \( n \)\nbd dimensional contexts of the form \( u \cpsub{k, \iota} - \) and \( - \subcp{k, \iota} u \) for \( 0 \le k < n \) and \( u \in \gr{k+1}{\Rd X} \).
\end{comm}

\begin{dfn} [Natural equivalence of round contexts]
    Let \( X \) be a diagrammatic set, and \( \F, \G \colon \Pd X(v, w) \to \Pd X(v', w') \) be round contexts.
    A family of equivalences \( \th a \colon \F a \celto \G a \) indexed by round diagrams \( a \colon v \celto w \) is a \emph{natural equivalence from \( \F \) to \( \G \)} if, for all round diagrams \( a, b \colon v \celto w \), there exists a natural equivalence from \( {\F_{a, b}-} \cp{} \th b \) to \( \th a \cp{} {\G_{a, b}-} \) as round contexts \( \Pd X(a, b) \to \Pd X(\F a, \G b) \).
    We write \( \F \simeq \G \) if there exists a natural equivalence from \( \F \) to \( \G \). 
\end{dfn}

\begin{rmk}
    By \cite[Proposition 3.24]{chanavat2024equivalences}, the relation \( \simeq \) is an equivalence relations on round contexts.
\end{rmk}

\begin{dfn} [Weak inverse of a round context]
    Let \( X \) be a diagrammatic set, and \( \F \colon \Pd X(v, w) \to \Pd X(v', w') \) be a weakly invertible round context.
    A \emph{weak inverse} of \( \F \) is a weakly invertible context \( \F^* \colon \Pd X(v', w') \to \Pd X(v, w) \) such that \( \F^* \F \simeq - \) and \( \F \F^* \simeq - \).
\end{dfn}

\begin{thm} \label{thm:weakly_invertible_context_have_weak_inverse}
    Every weakly invertible round context has a weakly invertible weak inverse.
\end{thm}
\begin{proof}
    See \cite[Theorem 5.22]{chanavat2024equivalences}.
\end{proof}

\begin{lem} \label{lem:closure_left_right_context}
    Let \( X \) be a diagrammatic set, \( A \subseteq \Rd X \) be a subset of round diagrams, \( u \colon a \celto b \), \( v \colon a' \celto b' \) be round diagrams in \( X \) of dimension \( > 1 \), \( \G \) be a weakly invertible round context on \( \Pd X(\bd{}{-}a, \bd{}{+} b) \), and \( \fun{T} \colon \Pd X(\G a, \G b) \to \Pd X(a', b' ) \) be a trim context.
    Then
    \begin{enumerate}
        \item if \( \fun{T} \) is left, \( v \in \Linv(\Pst A) \), and \( \fun{T}\G_{a, b}u \relcelto A v \), then \( u \in \Linv(\Pst A) \);
        \item if \( \fun{T} \) is right, \( v \in \Rinv(\Pst A) \), and \( v \relcelto A \fun{T}\G_{a, b}u \), then \( u \in \Rinv(\Pst A) \).
    \end{enumerate}
\end{lem}
\begin{proof}
    We prove only the first point, the second is dual. 
    Since \( \fun{T} \) is a left trim context, it is of the form \( - \cp{} r \), with \( r = (\cdots(\G b \subcp{} r_1) \subcp{} \cdots) \subcp{} r_{k} \) for some \( k \in \mathbb{N} \) and round diagrams \( (r_i)_{i = 1}^k \), and \( a' = \bd{}{-} (\G_{a, b} u \cp{} r) = \G a \).
    By assumption, \( v \in \Linv(\Pst(A)) \), so there exists a round diagram \( v^L \) such that \( v \cp{} v^L \relcelto{A} \un a'  \).
    By Theorem \ref{thm:weakly_invertible_context_have_weak_inverse}, we can consider a weak inverse \( \G^* \) of \( \G \) together with a natural equivalence \( \th \) from the identity to \( \G^*\G \). 
    We let also \( \th^* \) be a pointwise weak inverse of \( \th \).
    We define the round diagram \( u^L \eqdef (\th_b \subcp{} (r \cp{} v^L)) \cp{} \th^*_a \).
    Then we have
    \begin{align*}
        u \cp{} u^L &= \hspace{6pt} ((u \cp{} \th_b) \subcp{} (r \cp{} v^L)) \cp{} \th^*_a \\
                    &\simeq \hspace{6pt} ((\th_a \cp{} (\G^*\G)_{a, b} u) \subcp{} (r \cp{} v^L)) \cp{} \th^*_a && \text{by naturality of } \th \\
                    &=\hspace{6pt} (\th_a \subcp{} (\fun{T}\G_{a, b} u \cp{} v^L)) \cp{} \th^*_a \\
                    &\relcelto{A} (\th_a \subcp{} (v \cp{} v^L)) \cp{} \th^*_a && \text{since } \fun{T}\G_{a, b}u \relcelto{A} v \\
                    &\relcelto{A} (\th_a \subcp{} (\un a')) \cp{} \th^*_a && \text{since } v \cp{} v^L \relcelto{A} \un a' \\
                    &\simeq\hspace{6pt} \th_a \cp{} \th^*_a && \text{using a right unitor}\\
                    &\simeq\hspace{6pt} \un a.
    \end{align*}
    By Lemma \ref{lem:celto_A_is_congruence_preorder}, \( u \cp{} u^L \relcelto{A} \un a \), thus \( u \in \Linv(\Pst A) \).
    This concludes the proof.
\end{proof}

\subsection{Equivalences in Gray products}

\begin{lem} \label{lem:distribution_lower_gen_pasting_gray}
    Let \( X, Y \) be diagrammatic sets, \( \ell \in \mathbb{N} \), \( u \in \Pd X \), \( w \in \Rd X \) with \( n \eqdef \dim w \), and \( v \in \Pd Y \).
    \begin{enumerate}
        \item Suppose that a pasting at a submolecule \( w \cpsub{\iota} u \) is defined. Then 
            \begin{equation} 
                \bd{n + \ell}{+} ((w \cpsub{\iota} u) \gray v) = (w \gray \bd{\ell}{(-)^n}v) \gencp{n + \ell - 1} \bd{n + \ell}{+} (u \gray v).
            \end{equation}
        \item Suppose that a pasting at a submolecule \( u \subcp{\iota} w \) is defined. Then 
        \begin{equation} 
            \bd{n + \ell}{-} ((u \subcp{\iota} w) \gray v) = \bd{n + \ell}{-} (u \gray v) \gencp{n + \ell - 1} (w \gray \bd{\ell}{(-)^{n - 1}}v) .
        \end{equation}
    \end{enumerate}
\end{lem}
\begin{proof}
    We prove the first point only, the second uses a dual proof. 
    We recall the formulas of Proposition \ref{prop:formula_gray_product} which will be freely used throughout the proof. 
    We proceed by induction on \( \ell \), the treatment of the base case \( \ell = 0 \) and the inductive step \( \ell > 0 \) diverge only later in the proof.
    For \( \ell \geq 0 \), we have:
    \begin{align*}
        \bd{n + \ell}{+} &((w \cpsub{\iota} u) \gray v) = \bd{n + \ell}{+}\left((w \cpsub{\iota} u) \gray \bd{\ell}{(-)^n}v \right) \gencp{n + \ell - 1} \bd{n + \ell}{+} (\bd{n - 1}{+}(w \cpsub{\iota} u) \gray v) \\
        &= \bd{n + \ell}{+}\left((w \gray \bd{\ell}{(-)^n}v) \gencp{n + \ell - 1} (u \gray  \bd{\ell}{(-)^n}v) \right) \gencp{n + \ell - 1} \bd{n + \ell}{+}(\bd{n - 1}{+} u \gray v) \\
        &= \left((w \gray \bd{\ell}{(-)^n}v) \gencp{n + \ell - 1} \bd{n + \ell}{+} (u \gray  \bd{\ell}{(-)^n}v) \right) \gencp{n + \ell - 1} \bd{n + \ell}{+}(\bd{n - 1}{+} u \gray v),
    \end{align*}
    where the last step follows from Lemma \ref{lem:boundary_generalised_pasting_is_generalised_pasting}.
    Then
    \begin{equation*}
        \bd{n + \ell}{+} (u \gray v) = \bd{n + \ell}{+}(u \gray \bd{\ell}{(-)^n} v) \gencp{n + \ell - 1} \bd{n + \ell}{+}(\bd{n - 1}{+} u \gray v),
    \end{equation*}
    so that
    \begin{equation*}
        \bd{n + \ell}{+} ((w \cpsub{\iota} u) \gray v) = (w \gray \bd{\ell}{(-)^n}v) \cup \bd{n + \ell}{+} (u \gray v).
    \end{equation*}
    It remains to show that the right hand side of the previous equality is a generalised pasting at the \( (n + \ell - 1) \)\nbd boundary.
    More explicitly, the right hand side is the following pushout over \( X \):
    \begin{center}
        \begin{tikzcd}
            {\bd{n - 1}{+}w \gray \bd{\ell}{(-)^n}v} & {\bd{n + \ell}{+} (u \gray v)} \\
            {w \gray \bd{\ell}{(-)^n}v} & {(w \gray \bd{\ell}{(-)^n}v) \cup \bd{n + \ell}{+} (u \gray v).}
            \arrow[""{name=0, anchor=center, inner sep=0}, hook, from=1-1, to=1-2]
            \arrow[hook', from=1-1, to=2-1]
            \arrow[hook', from=1-2, to=2-2]
            \arrow[hook, from=2-1, to=2-2]
            \arrow["\lrcorner"{anchor=center, pos=0.125, rotate=180}, draw=none, from=2-2, to=0]
        \end{tikzcd}
    \end{center}
    We show that the hypothesis of generalised pasting at the \( (n + \ell - 1) \)\nbd boundary are satisfied for this pushout square.
    \begin{enumerate}
        \item We already have \( \bd{n - 1}{+} w \gray \bd{\ell}{(-)^n} v \submol \bd{n + \ell - 1}{+} (w \gray \bd{\ell}{(-)^n} v) \). Moreover, 
        \begin{align*}
            \bd{n - 1}{+} w \gray \bd{\ell}{(-)^n} v &\submol \bd{n - 1}{-} u \gray \bd{\ell}{(-)^n} v \\
            & \submol \bd{n + \ell - 1}{-} (u \gray \bd{\ell}{(-)^n} v)
            = \bd{n + \ell - 1}{-} (u \gray \bd{\ell}{(-)^{n - 2}} v) \\
            & \submol \bd{n + \ell - 1}{-} (u \gray v)
            = \bd{n + \ell - 1}{-} \bd{n + \ell}{+}(u \gray v).
        \end{align*}
        \item The diagram \( (w \gray \bd{\ell}{(-)^n}v) \cup \bd{n + \ell}{+} (u \gray v) \) is already the pasting diagram \( \bd{n + \ell}{+}((w \cpsub{\iota} u) \gray v) \), hence its input and output \( (n + \ell - 1) \)\nbd boundaries are again pasting diagrams.
        \item Finally, we have
        \begin{align*}
            \bd{n + \ell - 1}{-} (w \gray \bd{\ell}{(-)^n}v) &\submol \bd{n - 1}{-} w \gray \bd{\ell}{(-)^n}v
            \submol (\bd{n - 1}{-} (w \cpsub{\iota} u)) \gray \bd{\ell}{-(-)^{n - 1}}v \\
            &\submol \bd{n + \ell - 1}{-} ((w \cpsub{\iota} u) \gray v).
        \end{align*}
        Last, we use induction and globularity: if \( \ell = 0 \), then
        \begin{equation*}
            \bd{n - 1}{+} (u \gray v) = \bigcup_{i = 0}^{n - 1} \bd{i}{+} u \gray \bd{n - 1 - i}{(-)^i} v
            = \bigcup_{i = 0}^{n - 1} \bd{i}{+} (w \cpsub{\iota} u) \gray \bd{n - 1 - i}{(-)^i} v,
        \end{equation*}
        the latter term being equal to \( \bd{n - 1}{+} ((w \cpsub{\iota} u) \gray v) \).  
        Otherwise, \( \ell > 0 \) and inductively, the statement is true of \( \ell - 1 \), showing that
        \begin{align*}
            \bd{n + \ell - 1}{+} (u \gray v) &\submol (w \gray \bd{\ell - 1}{(-)^n}v) \gencp{n + \ell - 2} \bd{n + \ell - 1}{+} (u \gray v) \\
            &= \bd{n + \ell - 1}{+} ((w \cpsub{\iota} u) \gray v).
        \end{align*}
    \end{enumerate}
    All the hypothesis of generalised pasting are satisfied, this concludes the induction and the proof.
\end{proof}

\begin{lem} \label{lem:gray_preserve_left_right_pasting_contexts}
    Let \( X, Y \) be diagrammatic sets, \( u \in \Pd X \), \( w \in \Rd X \) with \( n \eqdef \dim w \), \( v \in \Pd Y \). For \( \ell \) ranging over \( \set{0, \ldots, \dim u + \dim v - n} \), call \( a_\ell \eqdef \bd{n + \ell}{-} (u \gray v) \), \( b_\ell \eqdef \bd{n + \ell}{+} (u \gray v) \).
    \begin{enumerate}
        \item Suppose that a pasting at a submolecule \( w \cpsub{\iota} u \) is defined, denote by \( \fun{R}^\ell \) the context determined by the subdiagram \( b_\ell \submol \bd{n + \ell}{+} ((w \cpsub{\iota} u) \gray v) \).
        Then we have inductively on \( \ell \in \set{0, \ldots, \dim u + \dim v - n} \),
        \begin{equation*}
            \begin{cases}
                \fun{R}^0 = (w \gray \bd{0}{(-)^n} v) \cpsub{} -, &\text{and} \\
                \fun{R}^{\ell} = (w \gray \bd{\ell}{(-)^n} v) \cpsub{} \fun{R}^{\ell - 1}_{a_{\ell - 1}, b_{\ell - 1}} & \text{if } \ell > 0.
            \end{cases}
        \end{equation*}
        \item Suppose that a pasting at a submolecule \( u \subcp{\iota} w \) is defined, denote by \( \fun{L}^\ell \) the context determined by the subdiagram \( a_\ell \submol \bd{n + \ell}{-} ((u \subcp{\iota} w) \gray v) \).
        Then we have inductively on \( \ell \in \set{0, \ldots, \dim u + \dim v - n} \),
        \begin{equation*}
            \begin{cases}
                \fun{L}^0 = - \subcp{} (w \gray \bd{0}{(-)^{n - 1}} v), &\text{and} \\
                \fun{L}^{\ell} =\fun{L}^{\ell - 1}_{a_{\ell - 1}, b_{\ell - 1}}  \subcp{} (w \gray \bd{\ell}{(-)^{n - 1}} v) & \text{if } \ell > 0.
            \end{cases}
        \end{equation*} 
    \end{enumerate}
\end{lem}
\begin{proof}
    Follows directly from the formulas of Lemma \ref{lem:formula_generalised_pasting} applied to Lemma \ref{lem:distribution_lower_gen_pasting_gray}. 
\end{proof}

\begin{lem} \label{lem:context_determined_by_boundary_gray}
    Let \( X, Y \) be diagrammatic sets, \( u \colon a \celto b \) be a pasting diagram in \( X \) with \( n \eqdef \dim u \), \( v \) be a pasting diagram in \( Y \).
    For \( \ell \) ranging over \( \set{0, \ldots, \dim v} \), call \( a_\ell \eqdef \bd{n + \ell}{-} (a \gray b) \), \( b_\ell \eqdef \bd{n + \ell}{+} (b \gray v) \), and denote respectively by \( \fun{R}^\ell \) and \( \fun{L}^\ell \) the contexts determined by the subdiagrams \(  b_\ell \submol \bd{n + \ell}{+} (u \gray v) \), and \( a_\ell \submol \bd{n + \ell}{-} (u \gray v) \).
    Then inductively on \( \ell \in \set{0, \ldots, \dim v} \),
    \begin{equation*}
        \begin{cases}
            \fun{R}^0 = (u \gray \bd{0}{(-)^n} v) \cpsub{} -, & \text{and} \\
            \fun{R}^\ell = (u \gray \bd{\ell}{(-)^n} v) \cpsub{} \fun{R}^{\ell - 1}_{a_{\ell - 1}, b_{\ell - 1}} & \text{if } \ell > 0,
        \end{cases} 
    \end{equation*}
    and dually,
    \begin{equation*}
        \begin{cases}
            \fun{L}^0 = - \subcp{} (u \gray \bd{0}{(-)^{n + 1}} v), & \text{and} \\
            \fun{L}^\ell = \fun{L}^{\ell - 1}_{a_{\ell - 1}, b_{\ell - 1}} \subcp{} (u \gray \bd{\ell}{(-)^{n + 1}} v)  & \text{if } \ell > 0.
        \end{cases}
    \end{equation*}
\end{lem}
\begin{proof}
    We prove the first point, the second is dual.
    When \( \ell = 0 \), \( \bd{n}{+} (u \gray v) \) is obtained by Proposition \ref{prop:formula_gray_product} as the generalised pasting
    \begin{center}
        \begin{tikzcd}
            {b \gray \bd{0}{(-)^n} v} & {\bd{n}{+}(b \gray v)} \\
            {u \gray \bd{0}{(-)^n} v} & {\bd{n}{+} (u \gray v).}
            \arrow[""{name=0, anchor=center, inner sep=0}, hook, from=1-1, to=1-2]
            \arrow[hook', from=1-1, to=2-1]
            \arrow[hook', from=1-2, to=2-2]
            \arrow[hook, from=2-1, to=2-2]
            \arrow["\lrcorner"{anchor=center, pos=0.125, rotate=180}, draw=none, from=2-2, to=0]
        \end{tikzcd}
    \end{center}
    But, \( b \gray \bd{0}{(-)^n} v = \bd{n - 1}{+} (u \gray \bd{0}{(-)^n} v) \), thus
    \begin{equation*}
        \bd{n}{+} (u \gray v) = (u \gray \bd{0}{(-)^n} v) \cpsub{} \bd{n}{+}(b \gray v)
    \end{equation*}
    hence \( \fun{R}^0 \) is \( (u \gray \bd{0}{(-)^n} v) \cpsub{} - \).
    For \( 0 < \ell \le \dim v \), we have by Proposition \ref{prop:formula_gray_product}
    \begin{equation*}
        \bd{n + \ell}{+} (u \gray v) = (u \gray \bd{\ell}{(-)^n} v) \gencp{n + \ell - 1} \bd{n + \ell}{+} (b \gray v),
    \end{equation*}
    thus we conclude inductively by globularity and Lemma \ref{lem:formula_generalised_pasting}.
\end{proof}

\begin{thm} \label{thm:equivalences_in_gray}
    Let \( X, Y \) be diagrammatic sets. 
    Then
    \begin{equation*}
        (\Eqv X \gray \Rd Y) \cup (\Rd X \gray \Eqv Y) \subseteq \Eqv (X \gray Y).
    \end{equation*}
\end{thm}
\begin{proof}
    We prove by induction on \( \ell \geq 0 \) that \( \Eqv X \gray \gr{\ell}{\Rd Y} \subseteq \Eqv (X \gray Y) \).
    The base case \( \ell = 0 \) is clear, since for each round diagram \( v \colon \pt \to Y \), the composite of \( X \cong X \gray \pt \) and \( - \gray v \colon X \gray \pt \to X \gray Y \) sends \( u \in \Rd X \) to \( u \gray v \), but morphisms of diagrammatic sets send equivalences to equivalences by Theorem \ref{thm:main_equivalences}.

    Inductively, let \( \ell > 0 \), and \( v \) be a round diagram in \( Y \) with \( \dim v = \ell \).
    We call \( A \eqdef \Eqv X \gray \set{v} \).
    We wish to show that \( A \subseteq \Eqv (X \gray Y) \).
    By Proposition \ref{prop:coinduction_on_closure_pst}, we will instead show that \( A \subseteq \B(\Pst A) \).
    Let \( u \colon a \celto b \) be an equivalence in \( X \) with \( n \eqdef \dim u \), and consider right and left inverses \( u^R \) and \( u^L \) of \( u \), together with equivalences \( h \colon \un b \celto u^R \cp{} u \) and \( z \colon u \cp{} u^L \celto \un a \).
    
    We let \( \fun{R} \) be the round context in \( X \gray Y \) determined by the subdiagram \( (u^R \cp{} u) \gray v \submol \bd{}{+} (h \gray v) \).
    Since \( h \gray v \) belongs to \( A \), we have 
    \begin{equation*}
        \bd{}{-} (h \gray v) \relcelto{A} \fun{R}((u^R \cp{} u) \gray v).
    \end{equation*}
    We wish to apply the second point of Lemma \ref{lem:closure_left_right_context} to deduce that \( (u^R \cp{} u) \gray v \) is in \( \Rinv(\Pst A) \).
    By inductive hypothesis on \( \ell \) with the equivalence \( h \) and Lemma \ref{lem:context_determined_by_boundary_gray}, the context \( \fun{R} \) is the composite of a right trim context and the promotion of a weakly invertible context. 
    By Proposition \ref{prop:formula_gray_product}, \( \bd{}{-} (h \gray v) \) is equal to the generalised pasting \( (\un b \gray v) \gencp{n + \ell - 1} (h \gray \bd{\ell - 1}{(-)^n} v) \).
    Since \( h \gray \bd{\ell - 1}{(-)^n} v \) is an equivalence by inductive hypothesis, and \( \un b \gray v \) is degenerate, so is an equivalence by Theorem \ref{thm:main_equivalences}, we deduce by Lemma \ref{lem:generalised_pasting_equivalences_is_equivalences} that \(  \bd{}{-} (h \gray v) \) is an equivalence, so belongs to \( \Rinv(\Pst A) \). 
    The hypothesis of Lemma \ref{lem:closure_left_right_context} are satisfied, hence \( (u^R \cp{} u) \gray v \) is in \( \Rinv(\Pst A) \).
    
    Now call \( \fun{R}' \) the round context in \( X \gray Y \) determined by the subdiagram \( u \gray v \submol (u^R \cp{} u) \gray v \).
    Since the relation \( \relcelto{A} \) is reflexive by Lemma \ref{lem:celto_A_is_congruence_preorder}, \( (u^R \cp{} u) \gray v \relcelto{A} \fun{R}' (u \gray v) \).
    We wish to apply the second point of Lemma \ref{lem:closure_left_right_context} again, and deduce that \( u \gray v \) is in \( \Rinv(\Pst A) \).
    By inductive hypothesis on \( \ell \) with the equivalence \( u^R \), and Lemma \ref{lem:gray_preserve_left_right_pasting_contexts}, \( \fun{R}' \) is the composite of a right trim context and the promotion of weakly invertible context.
    Since we established that \( (u^R \cp{} u) \gray v \in \Rinv(\Pst A) \), Lemma \ref{lem:closure_left_right_context} applies, hence \( u \gray v \in \Rinv(\Pst A) \).

    By a dual argument with \( z \) and \( u^L \) in place of \( h \) and \( u^R \), we obtain that \( u \gray v \in \Linv(\Pst A) \).
    Thus \( u \gray v \in \B(\Pst A) \).
    This proves that for all diagrammatic sets \( X, Y \), we have \( \Eqv X \gray \Rd Y \subseteq \Eqv (X \gray Y) \).
    Applied to \( \opp{Y} \) and \( \opp{X} \), we have \( \Eqv \opp{Y} \gray \Rd \opp{X} \subseteq \Eqv (\opp{Y} \gray \opp{X}) \).
    Applying the functor \( \opp{(-)} \) to this inclusion and using Lemma \ref{lem:equivalence_in_opposite} and Proposition \ref{prop:gray_opp_opp_gray}, we also get \( \Rd X \gray \Eqv Y \subseteq \Eqv (X \gray Y) \).
    This concludes the proof.
\end{proof}

\begin{lem} \label{lem:gray_product_preserves_A_contexts}
    Let \( X, Y \) be diagrammatic sets, \( A \subseteq \Rd X \), \( \F \) be a context on \( \Pd X(a, b) \), \( u \colon a \celto b \) be a round diagram in \( X \), and \( v \) be a round diagram in \( Y \).
    Then the context determined by the subdiagram \( u \gray v \submol \F u \gray v \) is an \( (A \gray \Rd Y) \)\nbd context on \( X \gray Y \).
\end{lem}
\begin{proof}
    We use the induction principle described in Comment \ref{comm:extended_pasting}.
    If \( \F \) is the identity context, then the context determined by \( u \gray v \submol \F u \gray v \) is the identity, so is an \( A \gray \Rd Y \)\nbd context.
    If \( \F \) is the composite \( \G \fun{H} \) of two \( A \)\nbd contexts \( \G \) and \( \fun{H} \), then the context determined by \( u \gray v \submol \F u \gray v \) is an \( (A \gray \Rd Y) \)\nbd context, since it is the composite of the contexts determined by
    \begin{equation*}
        u \gray v \submol \fun{H} u \gray v\text{, and } \fun{H}u \gray v \submol \G\fun{H}u \gray v,
    \end{equation*} 
    which are both \( (A \gray \Rd Y) \)\nbd contexts by inductive hypothesis on \( \fun{H} \) and \( \G \).
    Finally, if \( \F \) is of the form \( w \cpsub{\iota} - \) or \( - \subcp{\iota} w \) with \( w \in A \) such that \( \dim w \le \dim \F \), then this follows by applying inductively Lemma \ref{lem:gray_preserve_left_right_pasting_contexts}.
    This concludes the induction and the proof.
\end{proof}

\begin{lem} \label{lem:pushout_product_preserves_A_equations_gray}
    Let \( X, Y \) be diagrammatic sets, \( A \subseteq \Rd X\), \( \F \) be a round \( A \)\nbd context on \( \Pd X(a, b) \), \( u \colon a \celto b \) be a round diagram in \( X \), \( w \) be a round diagram in \( X \) parallel to \( \F u \), and \( v \) a round diagram in \( Y \).
    Then
    \begin{enumerate}
        \item for all round diagrams \( z \colon \F u \celto w \) with \( z \in A \), the context determined by the subdiagram \( u \gray v \submol \bd{}{-} (z \gray v) \) is an \( (A \gray \Rd Y) \)\nbd context on \( X \gray Y \); 
        \item for all round diagrams \( h \colon w \celto \F u \) with \( h \in A \), the context determined by the subdiagram \( u \gray v \submol \bd{}{+} (h \gray v) \) is an \( (A \gray \Rd Y) \)\nbd context on \( X \gray Y \).
    \end{enumerate}
\end{lem}
\begin{proof}
    Let \( z \colon \F u \celto w \) be a round diagram in \( A \). Then, the subdiagram \( u \gray v \submol \bd{}{-} (z \gray v) \) factors as the composite of the subdiagrams
    \begin{equation*}
        u \gray v \submol \F u \gray v\text{, and } (\bd{}{-} z) \gray v \submol \bd{}{-} (z \gray v),
    \end{equation*}
    By Lemma \ref{lem:gray_product_preserves_A_contexts}, the context determined by the subdiagram \( u \gray v \submol \F u \gray v \) is an \( (A \gray \Rd Y) \)\nbd context. 
    By inductively applying Lemma \ref{lem:context_determined_by_boundary_gray}, and since \( z \in A \), the context determined by the subdiagram \( (\bd{}{-} z) \gray v \submol \bd{}{-} (z \gray v) \) is an \( (A \gray \Rd Y) \)\nbd context.
    Thus the context determined by the subdiagram \( u \gray v \submol \bd{}{-} (z \gray v) \) is the composite of two \( (A \gray \Rd Y) \)\nbd context, hence is an \( (A \gray \Rd Y) \)\nbd context.
    Proceed dually for the second statement.
\end{proof}

\begin{lem} \label{lem:A_gray_Rd_context_is_A_gray_join_cell_context}
    Let \( X, Y \) be diagrammatic sets, \( A \subseteq \Rd X \) and \( \F \) be an \( (A \gray \Rd Y) \)\nbd context on \( X \gray Y \).
    Then \( \F \) is an \( (A \gray \cell Y) \)\nbd context on \( X \gray Y \).
\end{lem}
\begin{proof}
    We proceed by induction on the construction of \( \F \).
    The cases where \( \F \) is produced by identity, composition or promotion are straightforward.
    Suppose that \( \F \) is of the form \( (u \gray v) \cpsub{} -  \) with \( u \in A \) and \( v \in \Rd Y \).
    We prove by induction on the subdiagrams of \( v' \) of \( v \) that the contexts \( (u \gray v') \cpsub{} - \) are \( (A \gray \cell Y) \)\nbd contexts.
    The statement is clearly true of the subdiagrams \( v' \submol v \) of dimension \( 0 \), since all of these are cells.
    Suppose the statement is true of all proper subdiagrams of \( v \).
    Then, either \( v \) is a cell, and we are done, or decomposes as \( v = v_1 \cp{k} v_2 \).
    By Proposition \ref{prop:formula_gray_product} and Lemma \ref{lem:formula_generalised_pasting}, calling \( (i, j) \eqdef (1, 2) \) if \( \dim u \) is even, otherwise \( (i, j) \eqdef (2, 1) \), we have
    \begin{equation*}
        (u \gray (v_1 \cp{k} v_2)) \cpsub{} - = (u \gray v_i) \cpsub{} ((u \gray v_j) \cpsub{} -),
    \end{equation*}
    which is an \( (A \gray \cell Y) \)\nbd contexts, as the composite of two \( (A \gray \cell Y) \)\nbd contexts by inductive hypothesis.
    This ends the induction on subdiagrams of \( v \) and shows that \( (u \gray v) \cpsub{} -  \) is an \( (A \gray \cell Y) \)\nbd context. 
    The case where \( \F \) is of the form \( - \subcp{} (u \gray v) \) is dual.
    This concludes the proof.
\end{proof}

%% file: model.tex
\section{Model structures} \label{sec:model}

\subsection{Marked diagrammatic sets}

\begin{dfn} [Marked diagrammatic set]
    A \emph{marked diagrammatic set} is a pair of 
    \begin{enumerate}
        \item a diagrammatic set \( X \), and
        \item a set \( A \subseteq \gr{> 0}{\cell X } \) such that \( \dgn X \subseteq A \).
    \end{enumerate}
    A morphism \( f \colon (X, A) \to (Y, B) \) of marked diagrammatic sets is a morphism \( f \colon X \to Y \) of diagrammatic sets such that \( f(A) \subseteq f(B) \).
    This determines the category \( \mdgmSet \) of marked diagrammatic sets and morphisms.
\end{dfn}

\begin{rmk} \label{rmk:quasitopos_mdgmset}
    We recall by \cite[Proposition 2.27, Corollary 2.28]{chanavat2024model} that \( \mdgmSet \) is a locally presentable quasitopos. Furthermore, given a diagram \( \F \colon \cls J \to \mdgmSet \), a non-degenerate cell \( u \) is marked in \( \colim \F \) if and only if there exists \( j \in \cls J \), and a marked cell \( u_j \in \F j \) such that \( u \) is the image of \( u_j \) through the coprojection \( \iota_j \colon \F j \to \colim \F \).
\end{rmk}

\begin{dfn}
    There is forgetful functor \( \fun{U} \colon \mdgmSet \to \dgmSet \) obtained by forgetting the marking of a marked diagrammatic set.
    The functor \( \fun{U} \) has both a left and a right adjoint
    \begin{equation*}
        \minmark{(-)} \dashv \fun{U} \dashv \maxmark{(-)},
    \end{equation*}
    defined by \( \minmark{X} \eqdef (X, \dgn X) \) and \( \maxmark{X} \eqdef (X, \cell X) \).
    Since \( \dgn X \subseteq \eqv X \) by Theorem \ref{thm:main_equivalences}, there also is a functor \( \natmark{(-)} \colon \dgmSet \to \mdgmSet \) defined by \( \natmark{X} \eqdef (X, \eqv X) \).
\end{dfn}

\begin{dfn} [Entire monomorphism]
    We say that a monomorphism of marked diagrammatic sets is \emph{entire} if its underlying morphism of diagrammatic sets is an isomorphism.
    If \( i \colon (X', A') \to (X, A) \) is entire, the subset \( A \setminus i(A') \) is called the \emph{residual of \( i \)}.
\end{dfn}

\begin{dfn} [Marked regular directed complex]
	A \emph{marked regular directed complex} \( (P, A) \) is a regular directed complex \( P \) together with a set \( A \subseteq \gr{>0}{P} \) of marked elements.
	We will identify a marked regular directed complex \( (P, A) \) with the marked diagrammatic set whose set of marked cells is
    \begin{equation*}
        \dgn P \cup \set{ \mapel{x}\colon \imel{P}{x} \to P \mid x \in A }.
    \end{equation*}
	If \( P \) is a molecule or an atom, we speak of a \emph{marked molecule} or \emph{marked atom}.
	Given a molecule \( U \), we let \( \markmol{U} \) denote the marked molecule \( (U, \gr{>0}{\maxel{U}}) \). 
\end{dfn}

\begin{dfn} [Opposite of a marked diagrammatic set]
    Let \( (X, A) \) be a marked diagrammatic set.
    The \emph{opposite of \( (X, A) \)} is the marked diagrammatic set defined by \( \opp{(X, A)} \eqdef (\opp{X}, \opp{A}) \).
    This determines an involution \( \opp{(-)} \) on \( \mdgmSet \).
\end{dfn}  

\begin{dfn} [Gray product]
    Let \( (X, A), (Y, B) \) be marked diagrammatic sets.
    The \emph{Gray product of \( (X, A) \) and \( (Y, B) \)} is the marked diagrammatic set defined by
    \begin{equation*}
        (X, A) \gray (Y, B) \eqdef (X \gray Y, \dgn (X \gray Y) \cup A \gray \cell Y \cup \cell X \gray B).
    \end{equation*}
\end{dfn}

\noindent The Gray product extends to a monoidal structure on marked diagrammatic sets, with monoidal unit the terminal marked diagrammatic set. Furthermore, the functors \( \fun{U} \) and \( \minmark{(-)} \) are strict monoidal.

\begin{lem} \label{lem:gray_product_mdgmSet_biclosed}
    The Gray product of marked diagrammatic sets is a biclosed monoidal product.
\end{lem}
\begin{proof}
    An inspection of the definition together with Remark \ref{rmk:quasitopos_mdgmset} shows that \( - \gray - \) respects colimits in both variables. 
    By standard facts about locally presentable categories \cite[1.66]{adamek1994locally}, the monoidal product \( - \gray - \) has a right adjoint in both variables.
\end{proof}

\subsection{Diagrammatic \texorpdfstring{\( (\infty, n) \)}{(∞, n)}-categories}

\begin{dfn} [Cellular extension]
	Let \( X \) be a diagrammatic set.
	A \emph{cellular extension of \( X \)} is a diagrammatic set \( X_{S} \) together with a pushout diagram
    \begin{center}
        \begin{tikzcd}
            {\coprod_{e \in S} \bd{}{}U_e} &&& {\coprod_{u \in S} U_e} \\
            X &&& {X_S}
            \arrow["{(\bd{}{}e)_{e \in S}}", from=1-1, to=2-1]
            \arrow["{(e)_{e \in S}}", from=1-4, to=2-4]
            \arrow[hook, from=2-1, to=2-4]
            \arrow["{\coprod_{e \in S}\bd{U_e}{}}", hook, from=1-1, to=1-4]
            \arrow["\lrcorner"{anchor=center, pos=0.125, rotate=180}, draw=none, from=2-4, to=1-1]
        \end{tikzcd}
    \end{center}
    in \( \dgmSet \) such that \( U_e \) is an atom and \( \bd{U_e}{} \) is the inclusion of its boundary for each \( e \in S \).
	Each \( e \in S \) determines a cell \( e\colon e^- \celto e^+ \) in \( X_{S} \).
	In turn, the pushout is determined by the set of pairs of round diagrams \( \set{(e^-, e^+)}_{e \in S} \) in \( X \).
	We say that \( X_{S} \) is the result of \emph{attaching the cells \(\set{ e\colon e^- \celto e^+ }_{e \in S} \) to \( X \)}.
\end{dfn}

\begin{dfn} [Localisation of a diagrammatic set]
	Let \( (X, A) \) be a marked diagrammatic set.
    	We define \( \preloc{X}{A} \) to be the diagrammatic set obtained from \( X \) in the following two steps: for each cell \( a\colon u \celto v \) in \( A \cap \nd X \),
    \begin{enumerate}
        \item attach cells \( a^L \colon v \celto u \) and \( a^R \colon v \celto u \), then
        \item attach cells \( \hinv{L}(a) \colon a \cp{} a^L \celto \un u \) and \( \hinv{R}(a) \colon \un v \celto a^R \cp{} a \).
    \end{enumerate}
    Let \( \order{0}{X} \eqdef X \) and \( \order{0}{A} \eqdef A \).
    Inductively, for each \( n > 0 \), we let
    \begin{equation*}
    	\order{n}{X} \eqdef \preloc{\order{n-1}{X}}{\order{n - 1}{A}}, 
	\quad\quad 
	\order{n}{A} \eqdef \set{\hinv{L}(a), \hinv{R}(a) \mid a \in \order{n-1}{A}}.
    \end{equation*}
    We have a sequence \( (\order{n}{X} \incl \order{n+1}{X})_{n \geq 0} \) of inclusions of diagrammatic sets. 
    The \emph{localisation of \( X \) at \( A \)} is the colimit \( \loc{X}{A} \) of this sequence.

    For each cell \( a \) of shape \( U \) in \( A \), we define a family of cells \( \hinv{s}a \) of shape \( \hcyl{s}U \) in \( \loc{X}{A} \), indexed by strings \( s \in \set{L, R}^* \), by
	\begin{equation*}
        \begin{array}{lc}
            \;\hinv{\langle\rangle}a  \eqdef a & \\
            \hinv{Ls}a \eqdef
            \begin{cases}
                \hinv{L}(\hinv{s}a)
                    & \text{if \( a \in \nd X \)}, \\
                a \after \tau_s
                    & \text{if \( a \in \dgn X \)},
            \end{cases} &
            \hinv{Rs}a \eqdef 
            \begin{cases}
                \hinv{R}(\hinv{s}a)
                    & \text{if \( a \in \nd X \)}, \\
                a \after \tau_s
                    & \text{if \( a \in \dgn X \)},
            \end{cases}
        \end{array}
    \end{equation*}
	where \( \tau_s \) is the projection \( \hcyl{s}U \to U \).
	We also let, for each \( s \in \set{L, R}^* \),
    \begin{equation*}
        \hinv{s}^La \eqdef 
        \begin{cases}
            (\hinv{s}a)^L
                & \text{if \( a \in \nd X \)}, \\
            \rev{(a \after \tau_s)}
                & \text{if \( a \in \dgn X \)},
        \end{cases} \quad 
        \hinv{s}^Ra \eqdef 
        \begin{cases}
            (\hinv{s}a)^R
                & \text{if \( a \in \nd X \)}, \\
            \rev{(a \after \tau_s)}
                & \text{if \( a \in \dgn X \)}.
        \end{cases}
    \end{equation*}
\end{dfn} 

\noindent We will identify diagrams in \( X \) with their image through \( X \incl \loc{X}{A} \).
By construction, every cell in \( A \) becomes an equivalence in \( \loc{X}{A} \).
Recall by \cite[Proposition 2.139]{chanavat2024model} that there exists a colimit preserving functor
\begin{equation*}
    \Loc \colon \mdgmSet \to \dgmSet
\end{equation*}
sending \( (X, A) \) to \( \loc{X}{A} \), and a morphism \( f \colon (X, A) \to (Y, B) \) to the unique morphism \( \Loc f \) such that
\begin{enumerate}
    \item \( \Loc f \) restricts to \( f \) on \( X \incl \loc{X}{A} \),
    \item for each \( a \in A \) and \( s \in \set{L, R}^* \),
    \begin{equation*}
        \Loc f \colon \hinv{s}a \mapsto \hinv{s}f(a),\quad\quad \hinv{s}^La \mapsto \hinv{s}^Lf(a),\quad\quad \hinv{s}^Ra \mapsto \hinv{s}^Rf(a).    
    \end{equation*}
\end{enumerate}

\begin{dfn} [Walking equivalence]
    Let \( U \) be an atom of dimension \( > 0 \). 
    The \emph{walking equivalence of shape \( U \)} is the diagrammatic set \( \selfloc{U} \eqdef \Loc \markmol{U} \).
\end{dfn}
\noindent The walking equivalence of shape the arrow \( \arr \) is called the \emph{reversible arrow}, and is denoted by \( \revarr \). 

\begin{dfn} [Weak composite]
    Let \( X \) be a diagrammatic set, and \( u \colon U \to X \) be a round diagram of shape \( U \).
    A \emph{weak composite for \( u \)} is any cell \( \compos{u} \colon \compos{U} \to X \) such that \( u \simeq \compos{u} \).
\end{dfn}

\begin{dfn} [Diagrammatic \( (\infty, n) \)\nbd category]
    Let \( n \in \mathbb{N} \cup \set{\infty} \).
    We say that a diagrammatic set \( X \) is an \emph{\( (\infty, n) \)\nbd category} if:
    \begin{enumerate}
        \item every round diagram in \( X \) has a weak composite, and
        \item all cells of dimension \( > n \) in \( X \) are equivalences.
    \end{enumerate}
\end{dfn}

\begin{rmk}
    Any \( (\infty, n) \)\nbd category is by definition an \( (\infty, \infty) \)\nbd category.
\end{rmk}

\begin{lem} \label{lem:opp_of_infty_cat_is_infty_cat}
    Let \( n \in \mathbb{N} \cup \set{\infty} \), and \( X \) be an \( (\infty, n) \)\nbd category.
    Then \( \opp{X} \) is an \( (\infty, n) \)\nbd category.
\end{lem}
\begin{proof}
    Let \( \opp{u} \) be a round diagram in \( \opp{X} \).
    Then \( u \) has a weak composite \( \compos{u} \), and by Lemma \ref{lem:equivalence_in_opposite}, \( \opp{u} \simeq \opp{\compos{u}} \).
    Thus the cell \( \opp{\compos{u}} \) is a weak composite for \( \opp{u} \).
    By Lemma \ref{lem:equivalence_in_opposite} again, all cells of dimension \( > n \) in \( \opp{X} \) are equivalences.
    Thus \( \opp{X} \) is an \( (\infty, n) \)\nbd category. 
\end{proof}

\subsection{Definition of the model structures}

\begin{dfn} [Atomic horn]
    Let \( U \) be an atom, \( \dim U > 0 \).
    An \emph{atomic horn of \( U \)} is the data of
    \begin{enumerate}
        \item an element \( x \in \faces{}{\a} U \), for some \( \a \in \set{-, +} \),
        \item the inclusion \( \lambda^x_U \colon \Lambda^x_U \incl U \), where \( \Lambda^x_U \eqdef \bd{}{} U \setminus \clset{x} \).
    \end{enumerate}
\end{dfn}

\begin{dfn} [Context classified by an atomic horn]
    Let \( \Lambda^x_U \) be an atomic horn, let \( \a \in \set{-, +} \) such that \( x \in \faces{}{\a} U \), and \( f \colon \Lambda^x_U \to X \) a morphism of diagrammatic sets.
    Then the submolecule inclusion \( \iota \colon \clset{x} \submol \bd{}{\a} U \) determines a round context \( \F \) of shape \( \bd{}{\a} U \) on \( \Pd X(f(\bd{}{-} x), f(\bd{}{+}x)) \), called the \emph{context classified by \( f \)}.
\end{dfn}

\begin{dfn} [Marked horn]
    Let \( \lambda^x_U \colon \Lambda^x_U \to A \) be an atomic horn, let \( \top \) be the greatest element of \( U \), let \( \a \in \set{-, +} \) such that \( x \in \faces{}{\a} U \).
    We say that an inclusion
    \begin{equation*}
        \lambda^x_U \colon (\Lambda^x_U, A) \incl (U, A')
    \end{equation*}
    of marked regular directed complexes is a \emph{marked horn} if the context classified by \( \lambda^x_U \) is an \( A \)\nbd context in \( U \), and, moreover, 
    \begin{equation*}
        A' \eqdef 
        \begin{cases}
            A \cup \set{x, \top} & \text{if } \faces{}{-\a} U \subseteq A, \\
            A \cup \set{\top} & \text{otherwise.}
        \end{cases}
    \end{equation*}
\end{dfn}

\begin{dfn} [Marked walking equivalence]
    Let \( U \) be an atom of dimension \( > 0 \).
    The \emph{marked walking equivalence of shape \( U \)} is the marked diagrammatic set
    \begin{equation*}
        \selflocm{U} \eqdef (\selfloc{U}, \dgn \selfloc{U} \cup \set{U \incl \selfloc{U}}).
    \end{equation*}
\end{dfn}

\begin{lem} \label{lem:cellular_model_marked_dgmset}
    The sets of monomorphisms defined as
    \begin{align*}
        &M \eqdef \set{\minbdmap U \colon \minmark{(\bd{}{}U)} \incl \minmark{U} \mid U \in \Ob \atom} \cup \set{t_U \colon \minmark{U} \incl \markmol{U} \mid U \in \Ob \atom} \text{, and} \\
        &M' \eqdef\set{\minbdmap U \colon \minmark{(\bd{}{}U)} \incl \minmark{U} \mid U \in \Ob \atom} \cup \set{\markbdmap U \colon \minmark{(\bd{}{}U)} \incl \markmol{U} \mid U \in \Ob \atom} 
    \end{align*}
    are both cellular models for the monomorphisms of \( \mdgmSet \).
\end{lem}
\begin{proof}
    This is \cite[Lemma 3.20]{chanavat2024equivalences} for \( M \), and a straightforward variation of \cite[Theorem 1.3.8]{cisinski2019higher} for \( M' \). 
\end{proof}

\begin{dfn}
    Let \( n \in \mathbb{N} \cup \set{ \infty } \).
    We define the sets of monomorphisms of marked diagrammatic sets
    \begin{align*}
	    \Jhorn &\eqdef \set{\lambda_U^x \colon (\Lambda_U^x, A) \incl (U, A') \mid \text{$\lambda_U^x$ is a marked horn} }, \\
	    \Jn n &\eqdef \set{\minmark{U} \incl \markmol{U} \mid \text{$U \in \Ob\atom$, $\dim U > n$} }, \\
	    \Jloc &\eqdef \set{\minmark{\selfloc{U}} \incl \selflocm{U} \mid \text{$U \in \Ob\atom$, $\dim U > 0$}},
    \end{align*}
    and then let \( \Jcoind \) be the closure under isomorphism of \( \Jhorn \cup \Jn n \cup \Jloc \).
\end{dfn}

\begin{thm} \label{thm:model_structure_mdgmset}
    Let \( n \in \mathbb{N} \cup \set{\infty} \).
    There exists a model structure on marked diagrammatic sets, called the \emph{coinductive \( (\infty, n) \)\nbd model structure}, where
    \begin{enumerate}
        \item the cofibrations are the monomorphisms;
        \item \( \Jcoind \) is a pseudo-generating set of acyclic cofibrations;
        \item the fibrant objects are exactly the marked diagrammatic sets \( \natmark{X} \), for \( X \) an \( (\infty, n) \)\nbd category.
    \end{enumerate}
\end{thm}
\begin{proof}
    The model structure is constructed in \cite[3.24]{chanavat2024model}, with a pseudo-generating set of acyclic cofibrations given by \( \an(\Jcoind) \), which is a closure of \( \Jcoind \) under certain pushout-products with the pseudo-Gray tensor product defined in \cite[Definition 2.32]{chanavat2024model}.
    Then, the exact same proof as \cite[Lemma 4.8]{chanavat2024homotopy} can be used to show that \( \an(\Jcoind) = \Jcoind \).
    The fibrant objects are characterised in \cite[Theorem 4.9]{chanavat2024model}.
\end{proof}

\begin{comm}
    In \cite{chanavat2024model}, the authors construct another model structure on marked diagrammatic sets, called the \emph{inductive \( (\infty, n) \)\nbd model structure}, and show that it coincides with the coinductive \( (\infty, n) \)\nbd model structure for all finite values of \( n \).
    The present article is only concerned with the coinductive \( (\infty, n) \)\nbd model structure, hence also applies to the inductive \( (\infty, n) \)\nbd model structure if \( n \) is finite. 
    Theorem \ref{thm:model_structure_mdgmset} would also hold of the inductive \( (\infty, \infty) \)\nbd model structure, but we decided to omit the definition and proofs to avoid overwhelming the article. 
\end{comm}

\begin{thm} \label{thm:model_structure_dgmSet}
    Let \( n \in \mathbb{N} \cup \set{\infty} \).
    There exists a model structure on diagrammatic sets, called the \emph{\( (\infty, n) \)\nbd model structure} where:
    \begin{enumerate}
        \item the cofibrations are the monomorphisms;
        \item the fibrant objects are the \( (\infty, n) \)\nbd categories;
        \item a morphism \( f \) is a weak equivalence if and only if \( \minmark{f} \) is a weak equivalence. 
    \end{enumerate}
    Furthermore, the adjunction \( \minmark{(-)} \dashv \fun{U} \) forms a Quillen equivalence with the coinductive \( (\infty, n) \)\nbd model structure on marked diagrammatic sets. 
\end{thm}
\begin{proof}
    For the existence and the Quillen equivalence, see \cite[3.27 and Theorem 4.23]{chanavat2024model}.
    Since every diagrammatic set is cofibrant in the \( (\infty, n) \)\nbd model structure, the last point is given by Lemma \ref{lem:ken_brown_plus}.
\end{proof}

%% file: monoidal.tex
\section{The diagrammatic model structures are monoidal} \label{sec:monoidal}

\subsection{With respect to the Gray product}

In this section, we fix \( n \in \mathbb{N} \cup \set{\infty} \), and call \( M \) and \( M' \) the closure under isomorphisms of the two cellular models for marked diagrammatic sets given by Lemma \ref{lem:cellular_model_marked_dgmset}.

\begin{lem} \label{lem:pp_preserves_mono_dgmset}
    Let \( i \) and \( i' \) be two monomorphisms of diagrammatic sets.
    Then \( i \gray i' \) and \( i \pp{\gray} i' \) are monomorphisms of diagrammatic sets.
\end{lem}
\begin{proof}
    By \cite[Lemma 3.5, Remark 2.9]{chanavat2024homotopy} \( i \gray i' \) preserves monomorphisms and the set \( K \eqdef \set{\bd{U}{} \colon \bd{}{}U \incl U \mid U \in \Ob \atom} \) is a cellular model the monomorphisms of diagrammatic sets. 
    Then, by Lemma \ref{lem:pp_preserve_cof}, it is enough to show that for all atoms \( U \) and \( V \), the morphism \( \bd{U}{} \pp{\gray} \bd{V}{} \) is a monomorphism, but by inspection, this is the inclusion \( \bd{U \gray V}{} \colon \bd{}{} (U \gray V) \incl U \gray V \).
\end{proof}

\begin{lem} \label{lem:pp_preserves_mono_mdgmset}
    Let \( i, i' \) be two monomorphisms of marked diagrammatic sets. 
    Then \( i \gray i' \) and \( i \pp{\gray} i' \) are monomorphisms of marked diagrammatic sets.
\end{lem}
\begin{proof}
    Lemma \ref{lem:pp_preserves_mono_dgmset} together with the fact that the left adjoint monoidal functor \( \fun{U} \) preserves and reflects monomorphisms is enough to conclude. 
\end{proof}

\begin{lem} \label{lem:opposite_pp_product}
    Let \( f \colon (X, A) \to (Y, B) \) and \( g \colon (X', A') \to (Y', B') \) be morphisms of marked diagrammatic sets. 
    Then \( \opp{(f \pp{\gray} g)} \) is naturally isomorphic to \( \opp{g} \pp{\gray} \opp{f} \).
\end{lem}
\begin{proof}
    This follows directly from the fact that \( \opp{(-)} \) is an equivalence, so preserves colimits, and the natural isomorphism of Proposition \ref{prop:gray_opp_opp_gray}.
\end{proof}

\begin{lem} \label{lem:opp_Jhorn}
    The endofunctor \( \opp{(-)} \colon \mdgmSet \to \mdgmSet \) sends marked horns to marked horns.
\end{lem}
\begin{proof}
    By a straightforward induction, if \( \fun{F} \) is an \( A \)\nbd context in a diagrammatic set \( X \), then \( \opp{\F} \) is an \( \opp{A} \)\nbd context in \( \opp{X} \).
    Thus \( \opp{\Jhorn} = \Jhorn \).
\end{proof}

\begin{lem}\label{lem:pp_horn_is_horn}
    Let \( U, V \) be atoms, let \( x \in \faces{}{\a} U \) for some \( \a \in \set{-, +} \), and call \( \top \) the greatest element of \( V \).
    Then
    \begin{enumerate}
        \item \( \lambda^x_U \pp\gray \bd{V}{} = \lambda^{(x, \top)}_{U \gray V} \),
        \item \( \bd{V}{} \pp\gray \lambda^x_U = \lambda^{(\top, x)}_{V \gray U} \).
    \end{enumerate}
\end{lem}
\begin{proof}
    This is a particular case of \cite[Lemma 3.16]{chanavat2024homotopy}.
\end{proof}
    
\begin{prop} \label{prop:pp_marked_horn_is_marked_horn}
    We have \( \Jhorn \pp{\gray} M' \cup M' \pp{\gray} \Jhorn \subseteq \Jhorn \).
\end{prop}
\begin{proof}
    Let \( \lambda^x_U \colon (\Lambda^x_U, A) \incl (U, A') \) be a marked horn, let \( \a \in \set{-, +} \) such that \( x \in \faces{}{\a} U \), and let \( V \) be an atom.
    Consider first the case \( \lambda^x_U \pp{\gray} \minbdmap V \).
    By Lemma \ref{lem:pp_horn_is_horn}, \( \lambda^x_U \pp{\gray} \minbdmap V \) is the morphism of marked regular directed complexes
    \begin{equation*}
        \lambda^{(x, \top_V)}_{U \gray V} \colon (\Lambda^{(x, \top_V)}_U, B) \to (U \gray V, B'),
    \end{equation*}
    where \( B \eqdef A' \gray \bd{}{} V \cup A \gray V \), and \( B' \eqdef A' \gray V \).
    Thus, \( \lambda^{(x, \top_V)}_{U \gray V} \) classifies in \( U \gray V \) the context determined by the subdiagram \( \clset{x} \gray V \submol \bd{}{\a} (U \gray V) \), which is an \( (A \gray V) \)\nbd context by Lemma \ref{lem:pushout_product_preserves_A_equations_gray} and Lemma \ref{lem:A_gray_Rd_context_is_A_gray_join_cell_context}, thus a \( B \)\nbd context.
    Now we have \( B' \setminus B = (A' \setminus A) \gray \set{\top_V} \), as well as
    \begin{equation*}
        \faces{}{-\a} (U \gray V) = \faces{}{- \a} U \gray \set{\top_V} \cup \set{\top_U} \gray \faces{}{(-)^{\dim U + 1}\a} V. 
    \end{equation*}
    Thus either \( \faces{}{- \a} U \subseteq A \) and \( A' \setminus A = \set{\top_V, x} \), in which case, since \( \top_U \in A \), we have \( \faces{}{-\a} (U \gray V) \subseteq B \), and \( B' \setminus B = \set{(x, \top_V), \top_{U \gray V}} \);
    or \( \faces{}{- \a} U \) is not a subset of \( A \), thus \( \faces{}{-\a} (U \gray V) \) is not a subset of \( B \), and \( B' \setminus B = \set{\top_{U \gray V}} \).
    In both cases, we conclude that \( \lambda^x_U \pp{\gray} \minbdmap V \) is a marked horn.
    Similarly, \( \lambda^x_U \pp{\gray} \markbdmap V \) is the morphism of marked regular directed complexes
    \begin{equation*}
        \lambda^{(x, \top_V)}_{U \gray V} \colon (\Lambda^{(x, \top_V)}_U, C) \to (U \gray V, C'),
    \end{equation*}
    where \( C \eqdef A' \gray \bd{}{} V \cup A \gray V \cup \Lambda^x_U \gray \set{\top_V} \), and \( C' \eqdef A' \gray V \cup U \gray \set{\top_V} \).
    Since \( B \subseteq C \), \( \lambda^{(x, \top_V)}_{U \gray V} \) classifies a \( C \)\nbd context.
    Furthermore, \( \faces{}{-\a} (U \gray V) \subseteq C \), and since \( A' \setminus A \subseteq U \setminus \Lambda^x_U \), 
    \begin{equation*}
        C' \setminus C = (U \setminus \Lambda^x_U) \gray \set{\top_V} = \set{(x, \top_V), \top_{U \gray V}}.
    \end{equation*}
    Therefore, \( \lambda^x_U \pp{\gray} \markbdmap V \) is a marked horn.
    By the first part of the proof, Lemma \ref{lem:opposite_pp_product}, and Lemma \ref{lem:opp_Jhorn}, \( \opp{(\minbdmap V \pp{\gray} \lambda^x_U)} \) and \( \opp{(\markbdmap V \pp{\gray} \lambda^x_U)} \) are marked horns, thus by Lemma \ref{lem:opp_Jhorn} again, \( \minbdmap V \pp{\gray} \lambda^x_U \) and \( \markbdmap V \pp{\gray} \lambda^x_U \) are marked horns.
    This concludes the proof.
\end{proof}

\begin{lem} \label{lem:pp_with_entire_mono}
    Let \( i \colon (X, A') \incl (X, A) \) and \( j \colon (Y', B') \incl (Y, B) \) be monomorphisms of marked diagrammatic sets such that \( i \) is entire.  
    Then 
    \begin{enumerate}
        \item \( i \pp{\gray} j \) is entire on \( X \gray Y \) with residual given by the non-degenerate cells of \( (A \setminus A') \gray (\cell Y \setminus \cell jY') \);
        \item \( j \pp{\gray} i \) is entire on \( Y \gray X \) with residual given by the non-degenerate cells of \( (\cell Y \setminus \cell jY') \gray (A \setminus A') \).
    \end{enumerate}
\end{lem}
\begin{proof}
    We only prove the first point, the second is dual. 
    First, the underlying morphism of \( i \pp{\gray} j \) is an isomorphism and can be identified with \( \idd{X \gray Y} \), since the underlying morphism of \( i \) an isomorphism. 
    Then, by inspection of the definitions, together with Remark \ref{rmk:quasitopos_mdgmset}, the entire monomorphism \( i \pp{\gray}j \colon (X \gray Y, C) \to (X \gray Y, C') \) is such that
    \begin{equation*}
        C = A \gray \cell Y \cup \cell X \gray B \cup \dgn (X \gray Y),
    \end{equation*}
    and 
    \begin{equation*}
        C' = A \gray \cell jY' \cup A' \gray \cell Y \cup \cell X \gray B \cup \dgn (X \gray Y).
    \end{equation*}
    Thus a cell in \( C \) is either degenerate, hence is in \( C' \), or is non-degenerate of the form \( a \gray v \), and is either in \( C' \), or \( a \) is not in \( A' \) and \( v \) is not in \( \cell jY' \).     
    This concludes the proof.
\end{proof}

\begin{lem} \label{lem:recognize_acyclic_entire}
    Let \( i \) be an entire monomorphism of marked diagrammatic sets. 
    Suppose that \( i \) has the left lifting property against all marked diagrammatic sets \( \natmark{X} \), where \( X \) is an \( (\infty, n) \)\nbd category.
    Then \( i \) is an acyclic cofibration in the coinductive \( (\infty, n) \)\nbd model structure.
\end{lem}
\begin{proof}
    Since \( i \) is entire, any solution to a lifting problem against \( i \) is uniquely determined.
    By Theorem \ref{thm:model_structure_mdgmset}, marked diagrammatic sets of the form \( \natmark{X} \), with \( X \) an \( (\infty, n) \)\nbd category, are all the fibrant objects of the coinductive \( (\infty, n) \)\nbd model structure. 
    By Proposition \ref{prop:weak_equiv_wrt_cylinder}, the cofibration \( i \) is acyclic.
\end{proof}

\begin{thm} \label{thm:marked_model_structure_monoidal}
    The coinductive \( (\infty, n) \)\nbd model structure is monoidal with respect to the Gray product.
\end{thm}
\begin{proof}
    We show first that \( (\Jn n \cup \Jloc) \pp{\gray} M \cup M \pp{\gray} (\Jn n \cup \Jloc) \) is a set of acyclic cofibrations.
    Let \( V \) be an atom, and consider the entire monomorphism \( t_V \colon \minmark{V} \incl \markmol{V} \) in \( M \).
    By Lemma \ref{lem:pp_with_entire_mono}, the pushout-product of two entire monomorphisms is entire with an empty residual, thus is an isomorphism.
    Therefore, since any morphism \( j \in \Jn n \cup \Jloc \) is entire, \( t_V \pp{\gray} j \) and \( j \pp{\gray} t_V \) are acyclic cofibrations. 
    Let \( t_U \colon \minmark{U} \incl \markmol{U} \) in \( \Jn n \).
    Then \( t_U \pp{\gray} \minbdmap{V} \) is the entire monomorphism of marked diagrammatic sets 
    \begin{equation*}
        t \colon (U \gray V, \set{\top_U} \gray \bd{}{}V) \incl (U \gray V, \set{\top_U} \gray V).
    \end{equation*}
    Since \( \dim U > n \), also \( \dim (U \gray V) > n \), so that \( t_{U \gray V} \in \Jn n \). 
    Then, \( t \) is the pushout of the acyclic cofibration \( t_{U \gray V} \colon \minmark{(U \gray V)} \incl \markmol{(U \gray V)} \) along the entire inclusion \( \minmark{(U \gray V)} \incl  (U \gray V, \set{\top_U} \gray \bd{}{}V) \), so is an acyclic cofibration.
    Similarly, \( \minbdmap{V} \pp{\gray} t_U \) is an acyclic cofibration.
    Last, let \( j \colon \selfloc{U} \incl \selflocm{U} \) be in \( \Jloc \). 
    Since \( j \) is entire, by Lemma \ref{lem:pp_with_entire_mono}, \( j \pp{\gray} \minbdmap{V} \) is entire on \( \selfloc{U} \gray V \) whose only cell in its residual is \( u \gray \idd{V} \colon U \gray V \incl \selfloc{U} \gray V \), where \( u \colon U \incl \selfloc{U} \) is the evident cell.
    By construction, the cell \( u \colon U \incl \selfloc{U} \) is an equivalence in \( \selfloc{U} \), so by Theorem \ref{thm:equivalences_in_gray}, the cell \( u \gray \idd{V} \colon U \gray V \incl \selfloc{U} \gray V \) is an equivalence in \( \selfloc{U} \gray V \).
    Let \( X \) be an \( (\infty, n) \)\nbd category, and let \( f \colon (\selfloc{U} \gray V, C) \to \natmark{X} \) be a morphism of marked diagrammatic sets. 
    By Theorem \ref{thm:main_equivalences}, \( f(u \gray \idd{V}) \) is an equivalence in \( X \), hence is marked in \( \natmark{X} \). 
    Therefore, \( f \) extends along \( j \pp{\gray} \minbdmap{V} \).
    By Lemma \ref{lem:recognize_acyclic_entire}, \( j \pp{\gray} \minbdmap{V} \) is an acyclic cofibration.
    With a similar argument, \( \minbdmap{V} \pp{\gray} j \) is an acyclic cofibration.
    This proves that \( (\Jn n \cup \Jloc) \pp{\gray} M \cup M \pp{\gray} (\Jn n \cup \Jloc) \) is a set of acyclic cofibrations.
    This fact together with Proposition \ref{prop:pp_marked_horn_is_marked_horn} and Lemma \ref{lem:pp_preserve_cof} show that for any monomorphism \( i \) of marked diagrammatic sets, \( i \pp{\gray} \Jcoind \) and \( \Jcoind \pp{\gray} i \) are sets of acyclic cofibrations.
    Since \( \Jcoind \) is a pseudo-generating set of acyclic cofibrations by Theorem \ref{thm:model_structure_mdgmset}, and \( - \pp{\gray} - \) preserves monomorphisms, hence cofibrations, by Lemma \ref{lem:pp_preserves_mono_mdgmset}, we can conclude by Lemma \ref{lem:monoidal_pseudo_generating}.
\end{proof}

\begin{thm} \label{thm:gray_monoidal_dgmset}
    The \( (\infty, n) \)\nbd model structure on diagrammatic sets is monoidal with respect to the Gray product.
\end{thm}
\begin{proof}
    By Lemma \ref{lem:pp_preserves_mono_dgmset} and Theorem \ref{thm:model_structure_dgmSet}, the pushout-product of two cofibrations is a cofibration. 
    Next, the left adjoint \( \minmark{(-)} \colon \dgmSet \to \mdgmSet \) preserves pushouts, and is strict monoidal with respect to the Gray product.
    Therefore, for all cofibrations \( i \) and \( j \) of diagrammatic sets, \( \minmark{(i \pp{\gray} j)} \) is isomorphic to \( \minmark{i} \pp{\gray} \minmark{j} \).
    Now, \( \minmark{(-)} \) is a left Quillen equivalence, so if \( i \) or \( j \) is acyclic, Theorem \ref{thm:marked_model_structure_monoidal} implies that \( \minmark{(i \pp{\gray} j)} \) is an acyclic cofibration in the coinductive \( (\infty, n) \)\nbd model structure, thus by Lemma \ref{lem:ken_brown_plus}, the cofibration \( i \pp{\gray} j \) is acyclic.
\end{proof}

\subsection{Cylinder and opposite}

In this section, we fix \( n \in \mathbb{N} \cup \set{\infty} \). 

\begin{dfn} [Left and right marked cylinder]
    The \emph{left marked cylinder} is the endofunctor \( \markarr \gray - \) on marked diagrammatic sets, equipped with the natural transformations 
    \begin{equation*}
         (\iota^-, \iota^+) \colon \bigid{} \coprod \bigid{} \incl (\markarr \gray -),\quad\quad \sigma \colon (\markarr \gray -) \to \bigid{}.
    \end{equation*}
    induced by tensoring with the morphisms
    \begin{equation*}
        (0^+, 0^-) \colon \pt \coprod \pt \incl \markarr,\quad\quad \varepsilon \colon \markarr \to \pt.
    \end{equation*}
    We define dually the \emph{right marked cylinder} with the endofunctor \( - \gray \markarr \) instead.
\end{dfn}

\begin{prop} \label{prop:markcylinder_is_cylinder}
    The left and right marked cylinders are functorial cylinders for the coinductive \( (\infty, n) \)\nbd model structure. 
\end{prop}
\begin{proof}
    Consider the case of the left marked cylinder.
    In each component, the morphism \(  (\iota^-, \iota^+) \) is a cofibration, and \( \sigma (\iota^-, \iota^+) \) is a codiagonal.
    The inclusion \( 0^- \colon \pt \incl \markarr \) is a marked horn, thus a weak equivalence. 
    By two-out-of-three, \( \varepsilon \colon \markarr \to \pt \) is also a weak equivalence. 
    Let \( (X, A) \) be a marked diagrammatic set. 
    By Theorem \ref{thm:marked_model_structure_monoidal}, the functor \( - \gray (X, A) \) is left Quillen, so by Lemma \ref{lem:ken_brown_plus}, \( \varepsilon \gray (X, A) \colon \markarr \gray (X, A) \to \pt \gray (X, A) \), hence \( \markarr \gray (X, A) \to (X, A) \), are weak equivalences. 
    This shows that \( \markarr \gray - \) is a functorial cylinder.
    Proceed dually for the right marked cylinder.
\end{proof}

\begin{dfn} [Left and right reversible cylinder]
    The \emph{left reversible cylinder} is the endofunctor \( \locarr \gray - \) on diagrammatic sets, equipped with the natural transformations 
    \begin{equation*}
        (\iota^-, \iota^+) \colon \bigid{} \coprod \bigid{} \incl (\locarr \gray -),\quad\quad \sigma \colon (\locarr \gray -) \to \bigid{}.
   \end{equation*}
   induced by tensoring with the morphisms
   \begin{equation*}
       \Loc (0^+, 0^-) \colon \pt \coprod \pt \incl \locarr,\quad\quad \Loc \varepsilon \colon \locarr \to \pt.
   \end{equation*}
    We define dually the \emph{right reversible cylinder} with the endofunctor \( - \gray \locarr \) instead.
\end{dfn}

\begin{prop} \label{prop:reversible_cylinder_is_cylinder}
    The left and right reversible cylinders are functorial cylinders for the \( (\infty, n) \)\nbd model structure. 
\end{prop}
\begin{proof}
    Proceed as in the proof of Proposition \ref{prop:markcylinder_is_cylinder} using Theorem \ref{thm:gray_monoidal_dgmset} instead, and the fact that \( \Loc \varepsilon \) is a weak equivalence by \cite[Proposition 3.2]{chanavat2024model}.
\end{proof}

\begin{lem} \label{lem:opp_preserve_homotopies}
    Let \( W \) be an \( (\infty, n) \)\nbd category, and \( f, g \colon (X, A) \to \natmark{W} \) be morphisms of marked diagrammatic sets such that \( f \approx g \).
    Then \( \opp{f} \approx \opp{g} \).
\end{lem}
\begin{proof}
    Since \( \natmark{W} \) is fibrant in the coinductive \( (\infty, n) \)\nbd model structure, Proposition \ref{prop:markcylinder_is_cylinder} and Lemma \ref{lem:homotopic_implies_homotopic_any_cylinder} imply that we can choose a left homotopy \( \beta \colon \markarr \gray (X, A) \to \natmark{W} \) from \( f \) to \( g \). 
    Using the evident isomorphism \( \opp{(\markarr)} \cong \markarr \), we have by Proposition \ref{prop:gray_opp_opp_gray} the following commutative diagram
    \begin{center}
        \begin{tikzcd}
            {\opp{(X, A)}} \\
            \\
            {\opp{(X, A)} \gray \markarr \cong \opp{(X, A)} \gray \opp{(\markarr)}} && {\opp{(\natmark Z)}.} \\
            \\
            {\opp{(X, A)}}
            \arrow["{\iota^+}"', from=1-1, to=3-1]
            \arrow["{\opp f}", from=1-1, to=3-3]
            \arrow["{\opp \beta}", from=3-1, to=3-3]
            \arrow["{\iota^-}", from=5-1, to=3-1]
            \arrow["{\opp g}"', from=5-1, to=3-3]
        \end{tikzcd}
    \end{center}
    By Proposition \ref{prop:markcylinder_is_cylinder}, the right marked cylinder is a functorial cylinder for the coinductive \( (\infty, n) \)\nbd model structure, thus \( \opp{g} \approx \opp{f} \).
    Since all marked diagrammatic sets are cofibrant, we know by Remark \ref{rmk:left_homotopy_equivalence} that \( \approx \) is an equivalence relation, so we conclude that \( \opp{f} \approx \opp{g} \).
\end{proof}

\begin{thm} \label{thm:op_quillen_auto_eq_marked}
    The endofunctor \( \opp{(-)} \colon \mdgmSet \to \mdgmSet \) is a Quillen self-equivalence for the coinductive \( (\infty, n) \)\nbd model structure.
\end{thm}
\begin{proof}    
    The functor \( \opp{(-)} \) is an isomorphism of categories, and is its own adjoint with unit and counit being identities. 
    Let \( f \colon (X, A) \to (Y, B) \) be a weak equivalence. 
    We show that \( \opp{f} \) is a weak equivalence using Proposition \ref{prop:weak_equiv_wrt_cylinder}.
    Let \( W \) be an \( (\infty, n) \)\nbd category.
    By Lemma \ref{lem:equivalence_in_opposite} \( \opp{(\natmark{W})} = \natmark{(\opp{W})}\), which is fibrant in the coinductive \( (\infty, n) \)\nbd model structure by Lemma \ref{lem:opp_of_infty_cat_is_infty_cat}, so it has the right lifting property against \( f \), thus \( \natmark{W} \) has the right lifting property against \( \opp{f} \).
    Let \( u, v \colon (Y, B) \to \natmark{W} \) such that \( u \after \opp{f} \approx v \after \opp{f} \).
    By Lemma \ref{lem:opp_preserve_homotopies}, \( \opp{u} \after f \approx \opp{v} \after f \).
    Since \( f \) is a weak equivalence, \( \opp{u} \approx \opp{v} \), so by Lemma \ref{lem:opp_preserve_homotopies} again, \( u \approx v \).
    This shows that \( \opp{f} \) is a weak equivalence. 
    Since \( \opp{(-)} \) also preserves monomorphisms, it is left Quillen.
    Finally, since the unit and counit are identities, and \( \opp{(-)} \) is also right Quillen, this implies that the derived unit and counit are weak equivalences.
    This concludes the proof.
\end{proof}

\begin{thm} \label{thm:op_quillen_auto_eq}
    The endofunctor \( \opp{(-)} \colon \dgmSet \to \dgmSet \) is a Quillen self-equivalence for the \( (\infty, n) \)\nbd model structure.
\end{thm}
\begin{proof}
    Since \( \minmark{(-)} \) is a left Quillen equivalence and \( \opp{(\minmark{(-)})} \) is equal to \( \minmark{(\opp{(-)})} \), this is a simple combination of Lemma \ref{lem:ken_brown_plus} and Theorem \ref{thm:op_quillen_auto_eq_marked}.
\end{proof}